\documentclass[USenglish, a4paper, oneside, 10pt]{article}

\input{TeX/Preamble.tex}
\title{%
	Bregman level proximal subdifferentials and new characterizations of Bregman proximal operators%
	\thanks{%
		Z. Wang acknowledges the support of Xianfu Wang's NSERC Discovery Grants and the Mitacs Gloablink Award.
		The work of A. Themelis was supported by the JSPS KAKENHI grants number JP21K17710 and JP24K20737.
	}%
}
\author{%
	Ziyuan Wang\thanks{%
		Department of Mathematics, Irving K. Barber Faculty of Science,
		University of British Columbia, Kelowna, BC V1V 1V7, Canada.
		{\it E-mail:} {\sf ziyuan.wang96@outlook.com}
	}\and
	Andreas Themelis\thanks{%
		Faculty of Information Science and Electrical Engineering (ISEE), Kyushu University,
		744 Motooka, Nishi-ku, Fukuoka 819-0395, Japan.
		{\it E-mail:} {\sf andreas.themelis@ees.kyushu-u.ac.jp}
	}%
}
\date{}

\begin{document}

	\maketitle

	\begin{abstract}
		Classic subdifferentials in variational analysis may fail to fully represent the Bregman proximal operator in the absence of convexity.

In this paper, we fill this gap by introducing the left and right \emph{Bregman level proximal subdifferentials} and investigate them systematically.
Every Bregman proximal operator turns out to be the resolvent of a Bregman level proximal subdifferential under a standard range assumption, even without convexity.
Aided by this pleasant feature, we establish new correspondences among useful properties of the Bregman proximal operator, the underlying function, and the (left) Bregman level proximal subdifferential, generalizing classical equivalences in the Euclidean case.
Unlike the classical setting, asymmetry and duality gap emerge as natural consequences of the Bregman distance.
Along the way, we improve results by Kan and Song and by Wang and Bauschke on Bregman proximal operators.
We also characterize the existence and single-valuedness of the Bregman level proximal subdifferential, investigate coincidence results, and make an interesting connection to relative smoothness.
Abundant examples are provided to justify the necessity of our assumptions.
We also introduce \emph{anisotropic firm nonexpansiveness}, a new notion that is complementary to \emph{Bregman} firm nonexpansiveness and is shown to characterize relative smooth convex functions and convex envelopes via properties of gradient and proximal operators.

	\end{abstract}



	\phantomsection
	{%
		\renewcommand{\contentsname}{}%
		\footnotesize
		\pdfbookmark[1]{Contents}{toc}%
		\tableofcontents
	}%

	\section{Introduction}
		The proximal operator of a function \(\func{f}{\R^n}{\Rinf\coloneqq[-\infty,\infty]}\) for \(\lambda>0\) is
\[
	\Eprox_{\lambda f}:
	\begin{array}[t]{@{}r @{}c@{} l@{}}
		\R^n & {}\rightrightarrows{} & \R^n \\
		x & {}\mapsto{} & \displaystyle\argmin_{w\in\R^n}\set{f(w)+\tfrac{1}{2\lambda}\norm{x-w}^2}.
	\end{array}
\]
When \(f\) is convex, it is well known that the proximal operator is the resolvent of the Mordukhovich (limiting) subdifferential \(\ffunc{\lsubdiff f}{\R^n}{\R^n}\) of \(f\), that is,
\begin{equation}\label{eq:cvs resolvent}
	\Eprox_{\lambda f}=\left(\id+\lambda \lsubdiff f\right)^{-1},
\end{equation}
leading to beautiful correspondences among the properties of \(\Eprox_{\lambda f}\), the subdifferential \(\lsubdiff f\) and the function \(f\).
For instance, proximal operators of convex functions are resolvents of the associated subdifferential operators and therefore firmly nonexpansive---a fundamental property that underpins the convergence of various first-order methods in convex optimization; see \cite{bauschke2017convex}.
However, the representation \eqref{eq:cvs resolvent} fails in the absence of convexity with only \(\Eprox_{\lambda f}\subseteq(\id+\lambda\lsubdiff f)^{-1}\) being valid in full generality.
Rockafellar resolved this by introducing the \emph{level proximal subdifferential},\footnote{%
	We note that the level proximal subdifferential first appeared informally in \cite{rockafellar2021characterizing}, but no terminology was given there.
	Hence we adopt the convention used in \cite{wang2023every} for convenience.
}
denoted by \(\Epsubdiff f\), which is a refinement of the proximal subdifferential \(\Epsubdiff[]f\); see \cite[Eq. (2.13)]{rockafellar2021characterizing} or \cite[Thm. 3.7]{wang2023every}.
Regardless of convexity, proximal operators are \textit{always} resolvents of level proximal subdifferentials:
\begin{equation}\label{eq:level prox}
	\Eprox_{\lambda f}=\left(\id+\lambda \Epsubdiff f\right)^{-1},
\end{equation}
thereby various properties of the proximal operators in the absence of convexity can be characterized through the representation \eqref{eq:level prox}; see \cite{luo2024level,wang2023every}.
Assume now that \(\func{\kernel}{\R^n}{\Rinf}\) is proper, lsc, convex, and differentiable on \(\interior\dom\kernel\neq\emptyset\).
The Bregman distance induced by \(\kernel\) is \(\func{\D}{\R^n\times\R^n}{\Rinf}\) defined by
\begin{equation}
	\D(x,y)
	=
	\begin{cases}
		\kernel(x)-\kernel(y)-\innprod{\nabla\kernel(y)}{x-y} & \text{if } y\in\interior\dom\kernel
	\\
		\infty & \text{otherwise.}
	\end{cases}
\end{equation}
Since the Bregman distance is typically asymmetric, it gives rise to the left and the right Bregman proximal operators \(\ffunc{\prox_{\lambda f}}{\interior\dom\kernel}{\dom\kernel}\) and \(\ffunc{\prox*_{\lambda g}}{\dom\kernel}{\interior\dom\kernel}\) defined respectively by
\begin{align*}
	\prox_{\lambda f}(\bar y)
\coloneqq{} &
	\argmin_{x\in\dom\kernel}\set{f(x)+\tfrac{1}{\lambda}\D(x,\bar y)},
\\
	\prox*_{\lambda g}(\bar x)
\coloneqq{} &
	\argmin_{y\in\interior\dom\kernel}\set{g(y)+\tfrac{1}{\lambda}\D(\bar x,y)},
\end{align*}
for functions \(\func{f}{\dom\kernel}{\Rinf}\) and \(\func{g}{\interior\dom\kernel}{\Rinf}\), and points \(\bar y\in\interior\dom\kernel\) and \(\bar x\in\dom\kernel\).
In particular, \(\prox_{\lambda f}\) and \(\prox*_{\lambda g}\) reduce to \(\Eprox_{\lambda f}\) and \(\Eprox_{\lambda g}\) when \(\kernel=\tfrac{1}{2}\norm{{}\cdot{}}^2\).
The Bregman proximal operators are of central importance in optimization and variational analysis; see for instance the Chebyshev problem in the sense of the Bregman distance \cite{bauschke2009bregman}, various Bregman iterative methods \cite{bauschke2003iterating,bauschke2009bregman}, the celebrated mirror descent algorithms \cite{beck2003mirror,nemirovskij1983problem}, and Bregman algorithms for problems beyond traditional Lipschitz smoothness \cite{bauschke2017descent,bolte2018first,ahookhosh2021bregman,lu2018relatively,wang2024mirror,maddison2021dual}, just to name a few.
It is thus paramount to understand the properties of the Bregman proximal operators.
Significant progress has been made in \cite{bauschke2009bregman,kan2012moreau,laude2020bregman,bauschke2018regularizing}.
However, none of these works have addressed the following fundamental question:
\begin{equation}\label{Bregman level prox question}
	\text{\emph{What is the relation between \(\prox_{\lambda f}, \prox*_{\lambda g}\), and subdifferentials}?}
\end{equation}
Akin to the Euclidean case, answering \eqref{Bregman level prox question} shall shed light on how properties of the Bregman proximal operators correspond to those of the underlying function, therefore providing deep insights into these operators and their applications in numerical optimization.
When \(f\) is convex and its domain intersects \(\interior\dom\kernel\), the picture is clear owing to the optimality condition of convex functions, which yields the \emph{warped} resolvent characterization
\begin{equation}\label{eq:cvs warped}
	\prox_{\lambda f}=\left(\nabla\kernel+\lambda\lsubdiff f\right)^{-1}\circ\nabla\kernel
\end{equation}
that generalizes \eqref{eq:cvs resolvent}.
This may however fail in absence of convexity.

\emph{The goal of this paper is to resolve the fundamental question \eqref{Bregman level prox question} by introducing the left and right Bregman level proximal subdifferentials and investigating them systematically.}
Their resolvents \emph{always} coincide with the Bregman proximal operators under a necessary range assumption, regardless of convexity.
Equivalent descriptions of the Bregman level proximal subdifferentials are obtained.
We also investigate the single-valuedness and various coincidence results of the left Bregman level proximal subdifferential, and its connection to relative smoothness.
Correspondences among various properties of the Bregman proximal operator, the Bregman level proximal subdifferential, and the underlying function are established, including new characterizations and improvements of known results,
providing a clear picture about how different notions of a function, a subdifferential, and the warped resolvent of a subdifferential are related; see \cref{fig:diagram} for illustrations.
In particular, improvements of known results include adding new properties to the list of equivalences of relative weak convexity established by Kan and Song \cite[Thm. 4.2]{kan2012moreau} and proving that \(\prox_{\lambda f}\circ\nabla\kernel*\) being convex-valued implies relative weak convexity, while only the converse direction was proved by Wang and Bauschke \cite[Prop. 2.10(ii)]{wang2022bregman}.

We also introduce a notion of firm nonexpansiveness complementary to the Bregman one \cite{bauschke2003bregman,borwein2011characterization,wang2022bregman}, which we term \emph{anisotropic}, and employ it to characterize the gap between the convexity of a function and that of its envelope in terms of properties of its proximal mapping.
Furthermore, we show that relative smoothness of a convex function is equivalent to anisotropic firm nonexpansiveness of its gradient, under a mere Legendreness assumption on \(\kernel\), resolving an open question posed in \cite{themelis2026natural} as a simple byproduct.

\begin{figure}[hp]
	\centering
	\small
	\begin{subfigure}{\linewidth}%
		\centering
		\scalebox{0.85}{%
			\begin{tikzpicture}[diagram]
				\node[mybox, anchor=east] (E) at (0, 0) {
					\textbf{Euclidean setting (\cref{fact:Ehypo})}
				\\[5pt]
					\block{\(f+\lambda^{-1}\j\) is convex}
				\updown{\cite{rockafellar1998variational}}
					\block{\(\Eprox_{\lambda f}\) is maximally monotone}
				\updown{\cite{chen2020proximal}}
					\block{\(\Eprox_{\lambda f}\) is convex-valued}
				\updown{\cite{chen2020proximal}}
					\block{\(\Eprox_{\lambda f}=(\id+\lambda\lsubdiff f)^{-1}\)}
				};

				\node[mybox, anchor=west] (B) at (\blocksep, 0) {
					\textbf{Bregman setting (\cref{thm:hypo})}
				\\[5pt]
					\block{\(f+\lambda^{-1}\kernel\)} is convex
				\updown{\cite{kan2012moreau}}
					\block{\(\prox_{\lambda f}\circ\nabla\kernel*\)} is maximally monotone
				\updown[]{ \(\text{\cite{kan2012moreau}}+\text{\cite{wang2022bregman}}\)}%
					\block{\(\prox_{\lambda f}\)} is convex-valued
				\updown{}
					\block{\(\prox_{\lambda f}=(\nabla\kernel+\lambda\lsubdiff f)^{-1}\circ\nabla\kernel\)}
				};

				\draw[<-, myarrow] (E.east) -- (B.west) node[midway, above]{\(\kernel=j\)};
			\end{tikzpicture}
		}%
		\caption[Characterizations of weak convexity from Fact \ref*{fact:Ehypo} and Theorem \ref*{thm:hypo}]{%
			Characterizations of weak convexity from \cref{fact:Ehypo,thm:hypo}.
			The equivalence between relative weak convexity and maximal monotonicity was proved by Kan and Song \cite[Thm. 4.2]{kan2012moreau}.
			Wang and Bauschke \cite[Prop. 2.10]{wang2022bregman} showed that relative weak convexity implies convex-valued \(\prox_{\lambda f}\).
		}%
		\label{fig:hypo}%
	\end{subfigure}%

	\bigskip
	\begin{subfigure}{\linewidth}%
		\centering
		\scalebox{0.85}{%
			\begin{tikzpicture}[diagram]
				\node[mybox, anchor=south east] (Ef) at (0, 0) {
					\textbf{Euclidean setting (\cref{fact:classic})}
				\\[5pt]
					\block{\(\Epsubdiff f\) is maximally monotone}
				\updown{\cite{luo2024level}}
					\block{\(f\) is convex}
				\updown[\cite{rockafellar2021characterizing}]{\cite{rockafellar1976monotone}}
					\block{\(\norm{x_1-x_2}^2\leq\innprod{x_1-x_2}{y_1-y_2}\)}
				\\[0.5ex]
					\block{(\(\Eprox_{\lambda f}\) is FNE)}
				};

				\node[mybox, anchor=south west] (Bf) at (\blocksep, 0) {
					\textbf{Bregman setting (\cref{thm:DFNE})}
				\\[5pt]
					\block{\(\psubdiff f\) is maximally monotone on \(\interior\X\)}
				\updown{}
					\block{\(f\) is convex}
				\updown[]{\cite{bauschke2003bregman}}
					\block{\(\DD(x_1,x_2)\leq\innprod{x_1-x_2}{\nabla\kernel(y_1)-\nabla\kernel(y_2)}\)}
				\\[0.5ex]
					\block{(\(\prox_{\lambda f}\) is \B-FNE)}
				};

				\node[rotate=-90] (ass:range) at ($(Bf.east)-(0.2,0)$) {under \cref{ass:range}};

				\node[mybox, anchor=north east] (Eenv) at (0, -\blocksep) {
					\textbf{Euclidean setting (\cref{fact:classic})}
				\\[5pt]
					\block{\(\Eenv_\lambda{f}\) is convex}
				\updown{\cite{baillon1977quelques}}
					\block{\(\innprod{x_1-x_2}{y_1-y_2}\leq\norm{y_1-y_2}^2\)}
				\updown{}
					\block{\(\Eprox_{\lambda f}\) is FNE}
				};

				\node[mybox, anchor=north west] (Benv) at (\blocksep, -\blocksep) {
					\textbf{Bregman setting (\cref{thm:envcvx})}
				\\[5pt]
					\block{\(\env{f}\circ\nabla\kernel*\) is convex}
				\updown{}
					\block{\(\innprod{x_1-x_2}{\nabla\kernel(y_1)-\nabla\kernel(y_2)}\leq\DD(y_1,y_2)\)}%
				\updown{}
					\block{\(\prox_{\lambda f}\) is \a-FNE}%
				};

				\draw[myarrow, <- ] (Ef.east) -- ($(Ef.east)+(1,0)$) node[midway, above] {\(\kernel=\j\)};
				\draw[myArrow, <->] (Ef.south) -- (Eenv.north) node[midway, right] {\cite{wang2010chebyshev}};
				\draw[myArrow, <-, degil] ($(Benv.north)+(0.25,0)$) -- ($(Benv.north)+(0.25,1)$) node[midway, right, outer sep=5pt]{\cref{ex:f_not_env}};
				\draw[myArrow, ->, degil] ($(Benv.north)-(0.25,0)$) -- ($(Benv.north)+(-0.25,1)$) node[midway, left, outer sep=5pt]{\cref{ex:env_not_f}};
				\draw[myarrow, <- ] (Eenv.east) -- ($(Eenv.east)+(1,0)$) node[midway, above] {\(\kernel=\j\)};
			\end{tikzpicture}
		}%
		\caption[Characterizations of convexity of f and that of its envelope from Fact \ref*{fact:classic} and Theorems \ref*{thm:envcvx} and \ref*{thm:DFNE}]{%
			Characterizations of convexity of \(f\) and that of its envelope from \cref{fact:classic,thm:envcvx,thm:DFNE}.
			Here, \(x_i=\prox_{\lambda f}(y_i)\) for arbitrary \(y_i\in\interior\dom\kernel\), \(i=1,2\) (with \(\kernel=\j\) in the Euclidean case).
			The lack of equivalence between top- and bottom-right boxes rephrases a duality gap documented in \cite{laude2023dualities} in terms of different firm nonexpansiveness notions; see \cref{def:BFNE,def:aFNE}.
		}%
		\label{fig:cvx}%
	\end{subfigure}%
	\normalsize
	\caption[Equivalences, or lack thereof, for a Legendre and 1-coercive dgf phi and a proper and lsc function f]{%
		Equivalences, or lack thereof, for a Legendre and 1-coercive dgf \(\func{\kernel}{\R^n}{\Rinf}\), \(\lambda>0\), and a proper lsc \(\func{f}{\dom\kernel}{\Rinf}\) (right), compared to the Euclidean case (left).
	}%
	\label{fig:diagram}%
\end{figure}

The rest of the paper is organized as follows.
In \cref{sec:preliminaries} we recall elements from variational analysis, including technical lemmas regarding functions and operators defined on \emph{subsets} of \(\R^n\); this feature will enable the development of a more general and comprehensive theory.
In \cref{sec:psubdiff}, we define the \emph{Bregman level proximal subdifferentials}, investigate their properties, and provide characterizations.
\Cref{sec:global} is devoted to correspondences among various properties of a function, its Bregman proximal operator, and its Bregman level proximal subdifferential operator.
Finally, \cref{sec:conclusion} concludes the paper.
For convenience, a summary of the main notational conventions is provided in \cref{tab:notations}.

\begin{table}[ht]
	\renewcommand{\arraystretch}{1.5}%
	\setlength{\tabcolsep}{3pt}%
	\newlength\colI
	\newlength\colII
	\newlength\colIII
	\settowidth{\colI}{\(\func{\closure h\text{/}\conv h\text{/}\conv* h}{X}{\Rinf}\)}%
	\settowidth{\colII}{regular/limiting/Fenchel (convex) subdiff. of \(h\)}%
	\settowidth{\colIII}{\eqref{eq:rsubdiff}/\eqref{eq:lsubdiff}/\eqref{eq:fsubdiff}}%
	\begin{subtable}{\linewidth}
		\centering
		\begin{tabular}{@{}|c|c|c|@{}}
			\multicolumn{1}{@{}c}{\bf Object}
			&
			\multicolumn{1}{c}{\parbox{5em}{\centering\bf Definition}}
			&
			\multicolumn{1}{c@{}}{\parbox{5em}{\centering\bf Reference}}
		\\\hline
			\parbox{\colI}{\centering\(\X\) / \(\Y\) / \(\X*\) / \(\Y*\)}%
			&
			\parbox{\colII}{\centering\(\dom\kernel\) / \(\interior\dom\kernel\) / \(\dom\kernel*\) / \(\interior\dom\kernel*\)}%
			&
			\parbox{\colIII}{\centering\S\ref{sec:Bprox}}%
		\\\hline
			\parbox{\colI}{\centering\(\ffunc{\prox_{\lambda f}}{\Y}{\X}\)}%
			&
			\parbox{\colII}{\centering left Bregman proximal map of \(f\)}%
			&
			\multirow{2}{*}{\parbox{\colIII}{\centering \cref{def:prox}}}%
		\\
			\(\ffunc{\prox*_\lambda g}{\X}{\Y}\)
			&
			right Bregman proximal map of \(g\)
			&
		\\\hline
			\(\func{\env{f}}{\Y}{\Rinf}\)
			&
			left Bregman--Moreau envelope of \(f\)
			&
			\multirow{2}{*}{\cref{def:env}}
		\\
			\(\func{\env*_\lambda g}{\X}{\Rinf}\)
			&
			right Bregman--Moreau envelope of \(g\)
			&
		\\\hline
			\(\ffunc{\psubdiff f}{\X}{\R^n}\)
			&
			left Bregman level proximal subdiff. of \(f\)
			&
			\multirow{2}{*}{\cref{def:psubdiff}}
		\\
			\(\ffunc{\psubdiff*g}{\Y}{\R^n}\)
			&
			right Bregman level proximal subdiff. of \(g\)
			&
		\\\hline
		\end{tabular}
		\caption{%
			Bregman domains and objects defined for ``left'' \(\func{f}{\X}{\Rinf}\) and ``right'' functions \(\func{g}{\Y}{\Rinf}\) relative to a proper, convex, lsc dgf \(\func{\kernel}{\R^n}{\Rinf}\) that is differentiable on \(\interior\dom\kernel\neq\emptyset\), and a stepsize parameter \(\lambda>0\).%
		}%
	\end{subtable}
	\begin{subtable}{\linewidth}
		\centering
		\begin{tabular}{@{}|c|c|c|@{}}
			\multicolumn{1}{@{}c}{\bf Object}
			&
			\multicolumn{1}{c}{\parbox{5em}{\centering\bf Definition}}
			&
			\multicolumn{1}{c@{}}{\parbox{5em}{\centering\bf Reference}}
		\\\hline
			\(\ffunc{\rsubdiff f\text{/}\lsubdiff f\text{/}\fsubdiff f}{X}{\R^n}\)
			&
			regular/limiting/Fenchel (convex) subdiff. of \(f\)
			&
			\eqref{eq:rsubdiff}/\eqref{eq:lsubdiff}/\eqref{eq:fsubdiff}
		\\\hline
			\(\func{\closure f\text{/}\conv f\text{/}\conv* f}{X}{\Rinf}\)
			&
			lsc/convex/closed-convex hull of \(f\)
			&
			\eqref{eq:closure}/\S\ref{sec:cvx}/\S\ref{sec:cvx}
		\\\hline
			\(\ffunc{\tilde T}{\R^n}{\R^n}\)
			&
			canonical extension of \(T\)
			&
			\cref{def:Text}
		\\
			\(\func{\tilde f}{\R^n}{\Rinf}\)
			&
			canonical extension of \(f\)
			&
			\cref{def:fext}
		\\\hline
		\end{tabular}%
		\caption{%
			Nonsmooth and convex analysis notions adapted to functions \(\func{f}{X}{\Rinf}\) and operators \(\ffunc{T}{X}{Y}\) defined on nonempty convex sets \(X,Y\subseteq\R^n\).
		}%
	\end{subtable}
	\caption{%
		Synopsis of the adopted notation with references to the respective definitions.
	}%
	\label{tab:notations}%
\end{table}

	\section{Preliminaries}\label{sec:preliminaries}
		The set of natural numbers is \(\N\coloneqq\set{0,1,2,\dots}\), while \(\R\) and \(\Rinf\coloneqq\R\cup\set{\pm\infty}\) denote the set of real and extended-real numbers, respectively.
We use \(\innprod{{}\cdot{}}{{}\cdot{}}\) to denote the standard inner product on \(\R^n\), and let \(\norm{x}=\sqrt{\innprod{x}{x}}\) be the induced norm.
We also denote \(\j(x)=\frac{1}{2}\norm{x}^2\).
The identity function (on a space clear from context) is denoted \(\id\).
The interior, closure, and boundary of a set \(S\subseteq\R^n\) are respectively denoted as \(\interior S\), \(\overline{S}\), and \(\boundary S=\overline{S}\setminus\interior S\).
The \emph{indicator function} of \(S\) is \(\func{\indicator_S}{\R^n}{\Rinf}\) defined as \(\indicator_S(x)=0\) if \(x\in S\) and \(\infty\) otherwise.

In what follows, we let \(X\) and \(Y\) be arbitrary convex and nonempty subsets of \(\R^n\), and revisit fundamental notions from convex and variational analysis for functions \(\func{f}{X}{\Rinf}\) and set-valued operators \(\ffunc{T}{X}{Y}\) (as opposed to standard full-space objects \(\func{f}{\R^n}{\Rinf}\) and \(\ffunc{T}{\R^n}{\R^n}\)).
This more general approach allows us to embrace a much larger class of functions in our developments;
the interested reader is referred to the recent work \cite{themelis2026natural} for a more detailed discussion.

\subsection{Set-valued operators}
The notation \(\ffunc{T}{X}{Y}\) indicates a \emph{set-valued operator} \(T\) mapping points \(x\in X\) to sets \(T(x)\subseteq Y\).
The \emph{graph} of \(T\) is the set \(\graph T\coloneqq\set{(x,y)\in X\times Y}[y\in T(x)]\), while its \emph{(effective) domain} is \(\dom T\coloneqq\set{x\in X}[T(x)\neq\emptyset]\) and its \emph{range} is \(\range T\coloneqq\bigcup_{x\in X}T(x)\).
	\(T\) is said to be \emph{outer semicontinuous} (osc) if \(\graph T\) is closed relative to \(X\times Y\) (in the sense that \(\graph T=C\cap(X\times Y)\), where \(C\) is a closed subset of \(\R^n\times\R^n\)).
	Equivalently, \(T\) is osc if
	\[
		\set{y\in Y}[\exists\seq{x^k,y^k}\to(x,y) \text{ with }(x^k,y^k)\in\graph T\ \forall k]
	\subseteq
		T(x)
	\]
	holds for every \(x\in X\)
	(the set on the left-hand side is a \emph{Painlevé-Kuratowski outer limit} \cite[Def. 4.1]{rockafellar1998variational} in the space \(X\times Y\)).
\(T\) is \emph{locally bounded} if every \(\bar x\in X\) admits a neighborhood \(\mathcal N_{\bar x}\) (relative to \(X\)) such that \(\bigcup_{x\in\mathcal N_{\bar x}}T(x)\) is bounded.
A set-valued operator \(\ffunc{T}{X}{Y}\) is \emph{monotone} if \(\innprod{x-\bar x}{y-\bar y}\geq0\) for every \((x,y),(\bar x,\bar y)\in\graph T\), and is \emph{maximally monotone} if there exists no monotone operator \(\tilde{T}\) such that \(\graph T\subset\graph\tilde{T}\).\footnote{%
	Throughout, we use the symbols ``\(\subset\)'' and ``\(\supset\)'' do denote \emph{strict} inclusion.
}
Moreover, \(T\) is said to be \emph{firmly nonexpansive (FNE)} if \(\innprod{x-\bar x}{y-\bar y}\geq\norm{y-\bar y}^2\) for every \((x,y),(\bar x,\bar y)\in\graph T\).

The \emph{inverse} of \(T\) is always defined as the set-valued mapping \(\ffunc{T^{-1}}{Y}{X}\) given by \(T^{-1}(y)=\set{x\in X}[y\in T(x)]\) for every \(y\in Y\).
Of particular importance is the so-called \emph{warped resolvent of \(T\) with kernel \(K\)}, given by \((K+T)^{-1}\circ K\) where \(\func{K}{X}{Y}\) is a single-valued operator; see \cite[Def. 1.1]{bui2020warped}.
This notion generalizes the classic (non-warped) resolvent \((\id+T)^{-1}\), corresponding to \(K=\id\) (and \(X=Y=\R^n\)).

\subsection{Extended-real-valued functions}
The \emph{(effective) domain} of a function \(\func{f}{X}{\Rinf}\) is \(\dom f\coloneqq\set{x\in X}[f(x)<\infty]\), and \(f\) is said to be \emph{proper} if \(f\not\equiv\infty\) and \(f>-\infty\).
The \emph{epigraph} of \(f\) is \(\epi f\coloneqq\set{(x,\alpha)\in X\times\R}[f(x)\leq\alpha]\), and \(f\) is said to be \emph{lower semicontinuous} (lsc) if \(f(x)=\liminf_{x'\to x}f(x')\) holds for every \(x\in X\).\footnote{%
	Since \(f\) is a function \emph{defined} on \(X\), convergence \(x'\to x\) is meant relative to \(X\), so that writing \(\liminf_{x'\to x}f(x')\) is to be read as \(\liminf_{X\ni x'\to x}f(x')\).
	Similar conventions will be employed throughout.
}
Equivalently, \(f\) is lsc if \(\epi f\) is closed relative to \(X\times\R\).
The \emph{lower semicontinuous hull} of \(f\) is the lsc function \(\func{\closure f}{X}{\Rinf}\) defined as the pointwise supremum among all lsc functions \(X\to\Rinf\) majorized by \(f\), pointwise given by
\begin{equation}\label{eq:closure}
	\closure f(x)=\liminf_{x'\to x}f(x').
\end{equation}
Clearly, \(f\) is lsc iff \(f=\closure f\).

The \emph{Fréchet subdifferential} of \(f\) is \(\ffunc{\rsubdiff f}{X}{\R^n}\) given by
\begin{equation}\label{eq:rsubdiff}
	\rsubdiff f(x)
\coloneqq
	\set{v\in\R^n}[
		\textstyle
		\liminf_{\substack{z\to x\\z\neq x}}
		\frac{
			f(z)-f(x)-\innprod{v}{z-x}
		}{
			\norm{z-x}
		}
	\geq
		0
	],
\end{equation}
and we call an element \(v\in\rsubdiff f(x)\) a \emph{regular subgradient} of \(f\) at \(x\).
The \emph{Mordukhovich (limiting) subdifferential} of \(f\) is \(\ffunc{\lsubdiff f}{X}{\R^n}\) given by
\begin{equation}\label{eq:lsubdiff}
	\lsubdiff f(x)
\coloneqq
	\set{v\in\R^n}[
		\exists(x^k,v^k)\in\graph\rsubdiff f,
		~
		k\in\N: (x^k,f(x^k),v^k)\to (x,f(x),v)
	]
\end{equation}
for \(x\in\dom f\), and \(\lsubdiff f(x)=\emptyset\) otherwise.

\subsection{Convex analysis}\label{sec:cvx}%
The \emph{Fenchel (convex) subdifferential} of \(f\) is \(\ffunc{\fsubdiff f}{X}{\R^n}\), where
\begin{equation}
\label{eq:fsubdiff}
	\fsubdiff f(x)
\coloneqq
	\set{u\in\R^n}[f(x')\geq f(x)+\innprod{u}{x'-x} ~ \forall x'\in X].
\end{equation}
In particular,
\begin{equation}\label{eq:Fermat}
	\bar x\in\argmin f
\quad\Leftrightarrow\quad
	0\in\fsubdiff f(\bar x).
\end{equation}
The \emph{convex hull} of \(f\), denoted \(\conv f\), is the pointwise supremum among all convex functions \(X\to\Rinf\) majorized by \(f\), while the supremum of all such functions which are additionally lsc is the \emph{closed convex hull}, denoted \(\conv*f\).
In particular, one always has that
\begin{equation}\label{eq:convleq}
	\conv*f\leq\conv f\leq f
\quad\text{and}\quad
	\conv*f\leq\closure f\leq f.
\end{equation}
The convex conjugate of \(\func{h}{\R^n}{\Rinf}\) is \(\func{h^*}{\R^n}{\Rinf}\) given by \(h^*(\xi)=\sup_{x\in\R^n}\set{\innprod{x}{\xi}-h(x)}\), and we remind that \cite[Prop. 13.45]{bauschke2017convex}
\begin{equation}\label{eq:h**}
	\conv*h=h^{**}.
\end{equation}
We say that \(h\) is \emph{coercive} if \(\liminf_{\norm{x}\to\infty}h(x)=\infty\), and \emph{1-coercive} if \(\liminf_{\norm{x}\to\infty}h(x)/\norm{x}=\infty\).
Finally, throughout this work, we adopt a slight abuse of notation by defining the sum of \(\func{h}{\R^n}{\Rinf}\) and \(\func{f}{X}{\Rinf}\) as the function \(\func{(f+h)}{X}{\Rinf}\) given by \(f+h\coloneqq f+h\restr_X\), with similar conventions extending to analogous operations.

\begin{fact}[{\cite[Prop. 1.4.3]{hiriarturruty1996convex}}]\label{fact:fsubdiff}%
	Let \(\func{f}{X}{\Rinf}\) be proper and \(\bar x\in X\) be fixed.
	Consider the following statements:
	\begin{enumerateq}
	\item \label{fact:fsubdiff:1}%
		\(\fsubdiff f(\bar x)\neq\emptyset\);
	\item \label{fact:fsubdiff:2}%
		\(f(\bar x)=\conv* f(\bar x)\in\R\);
	\item \label{fact:fsubdiff:3}%
		\(\fsubdiff f(\bar x)=\lsubdiff(\conv* f)(\bar x)\).
	\end{enumerateq}
	One has \ref{fact:fsubdiff:1} \(\Rightarrow\) \ref{fact:fsubdiff:2} \(\Rightarrow\) \ref{fact:fsubdiff:3}.
\end{fact}

By combining \cref{fact:fsubdiff} with the inequalities in \eqref{eq:convleq} it is apparent that, for any \(\func{f}{X}{\Rinf}\) proper and \(x\in X\), one has that
\begin{equation}\label{eq:fsubdiffs}
	\fsubdiff f(x)\subseteq\fsubdiff[\closure f](x)\subseteq\fsubdiff[\conv*f](x).
\end{equation}

\subsection{Canonical extensions}
We now introduce a conventional way to extend operators and functions defined on subsets \(X\subseteq\R^n\) to full-space objects.

\begin{definition}[canonical extension of a set-valued operator]\label{def:Text}%
	The \emph{canonical extension} of \(\ffunc{T}{X}{Y}\) is the operator \(\ffunc{\tilde T}{\R^n}{\R^n}\) given by
	\[
		\tilde T(x)
	=
		\begin{cases}
			T(x) & \text{if } x\in X,\\
			\emptyset & \text{otherwise.}
		\end{cases}
	\]
\end{definition}

\begin{definition}[canonical extension of an extended-real-valued function]\label{def:fext}%
	The \emph{canonical extension} of \(\func{f}{X}{\Rinf}\) is the function \(\func{\tilde f}{\R^n}{\Rinf}\) given by
	\[
		\tilde f(x)
	=
		\begin{cases}
			f(x) & \text{if } x\in X,\\
			\infty & \text{otherwise.}
		\end{cases}
	\]
\end{definition}

Extensions defined in \cref{def:fext,def:Text} inherit many properties of the original functions/operators.
For instance, the canonical extension \(\ffunc{\tilde T}{\R^n}{\R^n}\) has the same domain, range, and graph of \(\ffunc{T}{X}{Y}\), while \(\func{\tilde f}{\R^n}{\Rinf}\) shares epigraph and effective domain with \(\func{f}{X}{\Rinf}\), and is convex if and only if \(\tilde f\) is.
Nevertheless, we note that the extension \(\tilde f\) may fail to be lsc (on \(\R^n\)) even when \(\func{f}{X}{\Rinf}\) is (on \(X\)), and similarly outer semicontinuity of \(T\) does not guarantee the same property for \(\tilde T\).
Therefore, special care is needed to deal with this discrepancy.
We keep such discussion minimal in this paper, and refer the readers to \cite{themelis2026natural} for a detailed discussion.

The following lemma clarifies the relation between the hulls of a function defined on \(X\subseteq\R^n\) and those of its canonical extension.

\begin{fact}[{\cite[Lem. 2.13]{themelis2026natural}}]\label{thm:fext}%
	Let \(\func{f}{X}{\Rinf}\) be proper and let \(\func{\tilde f}{\R^n}{\Rinf}\) be its canonical extension as in \cref{def:fext}.
	Then, the following hold:
	\begin{enumerate}
	\item \label{thm:clext}%
		\(\closure f=(\closure\tilde f)\restr_X\).
		In particular, \(f\) is lsc iff there exists an lsc (not necessarily proper) function \(\func{\hat f}{\R^n}{\Rinf}\) such that \(f=\hat f\restr_X\).
	\item \label{thm:cvxext}%
		\(
			\tilde f^{**}\restr_X=(\conv*\tilde f)\restr_X=\conv*f
		\leq
			\conv f=(\conv\tilde f)\restr_X
		\).
		The inequality holds as equality whenever \(\tilde f\) is 1-coercive.
	\item \label{thm:clcvx}%
		\(f\) is convex iff \(\tilde f\) is convex.
		In this case, one has that \(\conv*f=\closure f\).
	\end{enumerate}
\end{fact}

\begin{fact}[{\cite[Lem. 2.15]{themelis2026natural}}]\label{thm:argmincl}%
	For any \(\func{f}{X}{\Rinf}\) with \(X\subseteq\R^n\) nonempty, one has that \(\inf f=\inf\closure f=\inf\conv*f\) and \(\argmin f\subseteq\argmin\closure f\subseteq\argmin\conv*f\).
\end{fact}

\subsection{Bregman proximal map and Moreau envelope}\label{sec:Bprox}
Throughout the paper, we will consider a function \(\kernel\) complying with the following minimal working assumption, assumed throughout without further mention; additional requirements will be invoked when needed.
\begin{myframe}
	\centering
	\(\func{\kernel}{\R^n}{\Rinf}\) is a \emph{distance-generating function (dgf)}, namely, proper, lsc, convex, and differentiable on \(\interior\dom\kernel\neq\emptyset\).
\end{myframe}
Moreover,
\begin{myframe}
	\centering
	we denote \(\X\coloneqq\dom\kernel\) and \(\Y\coloneqq\interior\dom\kernel\), and similarly \(\X*\coloneqq\dom\kernel*\) and \(\Y*\coloneqq\interior\dom\kernel*\).
\end{myframe}

The Bregman distance induced by \(\kernel\) is \(\func{\D}{\R^n\times\R^n}{\Rinf}\) defined by
\begin{equation}\label{eq:D}
	\D(x,y)
=
	\begin{cases}
		\kernel(x)-\kernel(y)-\innprod{\nabla\kernel(y)}{x-y} & \text{if } y\in\interior\dom\kernel
	\\
		\infty & \text{otherwise,}
	\end{cases}
\end{equation}
and we denote the \emph{symmetrized} version as \(\DD(x,y)\coloneqq\D(x,y)+\D(y,x)\), namely
\begin{equation}\label{eq:DD}
	\DD(x,y)
=
	\begin{cases}
		\innprod{\nabla\kernel(x)-\nabla\kernel(y)}{x-y} & \text{if } x,y\in\interior\dom\kernel
	\\
		\infty & \text{otherwise\color{black}.}
	\end{cases}
\end{equation}

\begin{fact}[three-point identity {\cite[Lem. 3.1]{chen1993convergence}}]\label{thm:3p}%
	For \(x\in\dom\kernel\) and \(y,z\in\interior\dom\kernel\) it holds that
	\[
		\D(x,z)=\D(x,y)+\D(y,z)+\innprod{x-y}{\nabla\kernel(y)-\nabla\kernel(z)}.
	\]
\end{fact}
\begin{definition}[Legendre function]%
		We say that \(\kernel\) is \emph{of Legendre type} (or simply \emph{Legendre}) if it is
		\begin{enumerate}
		\item
			\emph{essentially smooth}, namely differentiable on \(\interior\dom\kernel\neq\emptyset\) and such that \(\norm{\nabla\kernel(x^k)}\to\infty\) whenever \(\interior\dom\kernel\ni x^k\to x\in\boundary\dom\kernel\), and
		\item
			\emph{essentially strictly convex}, namely strictly convex on every convex subset of \(\dom\lsubdiff\kernel\).
		\end{enumerate}
	\end{definition}

\begin{fact}[{\cite[Thm. 26.5]{rockafellar1970convex}} and {\cite[Thm. 3.7(v)]{bauschke1997legendre}}]\label{thm:D*}%
	Function \(\kernel\) is of Legendre type iff its conjugate \(\kernel*\) is.
	In this situation, one has that \(\func{\nabla\kernel}{\interior\dom\kernel}{\interior\dom\kernel*}\) is a bijection with inverse \(\nabla\kernel*\), and the following identity holds:
	\[
		\D(x,y)
	=
		\D*(\nabla\kernel(y),\nabla\kernel(x))
	\quad
		\forall x,y\in\interior\dom\kernel.
	\]
\end{fact}

Now we define the Bregman proximal maps and the Bregman--Moreau envelope of functions \(\func{f}{\X}{\Rinf}\) and \(\func{g}{\Y}{\Rinf}\).
Although functions \(\func{f,g}{\R^n}{\Rinf}\) are usually considered in the literature, see for instance \cite{kan2012moreau,wang2022bregman}, our approach takes into account a broader class of functions while not affecting the validity of known results;
see \cref{ex:ln} for an illustration and \cite[\S3]{themelis2026natural} for a more detailed discussion.
All these definitions are relative to a \emph{stepsize} parameter \(\lambda>0\).

\begin{definition}[Bregman proximal mappings]\label{def:prox}%
	The \emph{left (Bregman) proximal mapping} of \(\func{f}{\X}{\Rinf}\) is \(\ffunc{\prox_{\lambda f}}{\Y}{\X}\) defined by
	\begin{subequations}
		\begin{align}\label{eq:prox}
			\prox_{\lambda f}(\bar y)
		\coloneqq{} &
			\argmin_{x\in\X}\set{f(x)+\tfrac{1}{\lambda}\D(x,\bar y)}.
		\intertext{%
			The \emph{right (Bregman) proximal mapping} of \(\func{g}{\Y}{\Rinf}\) is \(\ffunc{\prox*_{\lambda g}}{\X}{\Y}\), where%
		}
			\prox*_{\lambda g}(\bar x)
		\coloneqq{} &
			\argmin_{y\in\Y}\set{g(y)+\tfrac{1}{\lambda}\D(\bar x,y)}.
		\end{align}
	\end{subequations}
\end{definition}

\begin{definition}[Bregman--Moreau envelopes]\label{def:env}%
	The \emph{left (Bregman--Moreau) envelope} of \(\func{f}{\X}{\Rinf}\) is \(\func{\env{f}}{\Y}{\Rinf}\) defined by
	\begin{subequations}
		\begin{align}
			\env{f}(\bar y)
		\coloneqq{} &
			\inf_{x\in\X}\set{f(x)+\tfrac{1}{\lambda}\D(x,\bar y)}.
		\intertext{%
			The \emph{right (Bregman--Moreau) envelope} of \(\func{g}{\Y}{\Rinf}\) is \(\func{\env*{g}}{\X}{\Rinf}\), where%
		}
			\env*{g}(\bar x)
		\coloneqq{} &
			\inf_{y\in\Y}\set{g(y)+\tfrac{1}{\lambda}\D(\bar x,y)}.
		\end{align}
	\end{subequations}
\end{definition}

\begin{definition}[\(\kernel\)-prox-boundedness {\cite[Def. 2.3]{kan2012moreau}}]\label{def:PB}%
	We say that \(\func{f}{\X}{\Rinf}\) is \emph{\(\kernel\)-prox-bounded} if there exist \(\lambda>0\) and \(y\in\Y\) such that \(\env{f}(y)>-\infty\).
	The supremum \(\pb\) of all such \(\lambda\) is the \emph{\(\kernel\)-prox-boundedness threshold} of \(f\).
\end{definition}

The main consequences of \(\kernel\)-prox-boundedness are listed in \cref{thm:mainprop}.
Before that, the following preliminary result reassures us that the referenced proofs can correctly be invoked in our more general setting, even for those functions \(\func{f}{\X}{\Rinf}\) which are proper and lsc on \(\X\) but may not possess an extension preserving both properties on the entire space \(\R^n\), such as the one in \cref{ex:ln}.

\begin{fact}[{\cite[Lem. 3.9]{themelis2026natural}}]\label{thm:extlsc}%
	Let \(\func{f}{\X}{\Rinf}\) and extend it to \(\func{\tilde f}{\R^n}{\Rinf}\) by setting it to \(\infty\) on \(\R^n\setminus\X\) as in \cref{def:fext}.
	Then, for any \(y\in\Y\) and \(\lambda>0\) one has
	\[
		\prox_{\lambda f}(y)=\prox_{\lambda\tilde f}(y)
	\quad\text{and}\quad
		\env{f}(y)=\env{\tilde f}(y).
	\]
	When \(f\) is \(\kernel\)-prox bounded, then the following also hold for any \(\lambda\in(0,\pb)\):
	\begin{enumerate}
	\item
		\textup{\cite[Prop. 2.4(i)]{wang2022bregman}}
		If \(\kernel\) is 1-coercive, then \(\lambda\tilde f+\kernel\) too is 1-coercive.
	\item \label{thm:extlsc:lsc}%
		If \(f\) is lsc (on \(\X\)), then \(\lambda\tilde f+\kernel\) too is lsc (on \(\R^n\)).
	\end{enumerate}
\end{fact}

\begin{fact}[{\cite[Thm. 2.2 and 2.6, Cor. 2.2]{kan2012moreau}}]\label{thm:mainprop}%
	Suppose that \(\kernel\) is 1-coercive, and that \(\func{f}{\X}{\Rinf}\) is proper, lsc, and \(\kernel\)-prox-bounded.
	Then, for any \(\lambda\in(0,\pb)\) the following hold:
	\begin{enumerate}
	\item
		\(\dom\env{f}=\dom\prox_{\lambda f}=\Y\).
	\item
		\(\func{\env{f}}{\Y}{\R}\) is continuous.
	\item \label{thm:mainprop:osc}%
		\(\ffunc{\prox_{\lambda f}}{\Y}{\X}\) is compact-valued and \(\graph\prox_{\lambda f}\) is closed in \(\Y\times \R^n\).
	\end{enumerate}
\end{fact}

We note that the properties listed above are proven in \cite{kan2012moreau} for functions defined and assumed lsc \emph{on the whole space \(\R^n\)}.
The more general statement considered here is taken from \cite[Fact 3.11]{themelis2026natural}, where lower semicontinuity merely on \(\X\) is shown to suffice.
The key observation behind this generalization lies in \cref{thm:extlsc}.
The advantages of this generalization is not the focus of this work, therefore we only illustrate an instance with the example below and refer the interested readers to \cite{themelis2026natural} for a comprehensive discussion on this topic.

\begin{example}\label{ex:ln}%
	Let \(\func{\kernel}{\R}{\Rinf}\) be defined as \(\kernel(x)=-\ln(x)\) for \(x\in\X=(0,\infty)\) and \(\kernel(x)=\infty\) for \(x\leq 0\).
	Then, \(\func{f}{\X}{\Rinf}\) given by \(f(x)=\ln(x)\) is proper, lsc, and \(\kernel\)-prox-bounded with \(\pb=1\).
	Although \(f\) does not admit a proper and lsc extension on \(\R\), \(\prox_{\lambda f}\) and \(\env{f}\) still enjoy the properties stated in \cref{thm:mainprop} for any \(\lambda\in(0,\pb)\).
\end{example}

\begin{definition}[Bregman proximal hull]%
	The \emph{left Bregman proximal hull} of \(\func{f}{\X}{\Rinf}\) is the function \(\func{\hull_\lambda{f}}{\X}{\Rinf}\) given by \(\hull_\lambda{f}\coloneqq-\env*{(-\env{f})}\).
\end{definition}

We end this section with several technical identities.
Some of these reduce to known formulas from \cite[Fact 2.6 \& Prop. 2.14]{wang2022bregman} when additional assumptions are met.

\begin{fact}[{\cite[Lem. 3.14]{themelis2026natural}}]\label{thm:left*}%
	Given \(\func{f}{\X}{\Rinf}\), extend it to \(\func{\tilde f}{\R^n}{\Rinf}\) by setting it to \(\infty\) on \(\R^n\setminus\X\) as in \cref{def:fext}.
	Then, for any \(\lambda>0\) the following hold:
	\begin{enumerate}
	\item \label{thm:envconj}%
		\(
			\lambda\env{f}
		=
			\bigl[\kernel*-(\lambda\tilde f+\kernel)^*\bigr]\circ\nabla\kernel
		\).
		In particular, \(-\env{f}\) is lsc.
		Moreover, when \(\kernel\) is of Legendre type,
		\(
			\lambda\env{f}\circ\nabla\kernel*
		=
			\bigl[\kernel*-(\lambda\tilde f+\kernel)^*\bigr]\restr_{\Y*}
		\).

	\item \label{thm:hull}%
		\(
			\begin{array}[t]{@{}r@{}l@{}}
				\lambda\hull_\lambda{f}
			=
				\bigl((\lambda\tilde f+\kernel)^*+\indicator_{\range\nabla\kernel}\bigr)^*
				\restr_{\X}
				-
				\kernel
			\leq{} &
				\conv*(\lambda f+\kernel)
				-
				\kernel
			\\
			\leq{} &
				\conv(\lambda f+\kernel)
				-
				\kernel
			\leq
				\lambda f.
			\end{array}
		\)\newline
		If \(\kernel\) is essentially smooth and 1-coercive, then the first inequality holds as equality.
		If instead \(\kernel\) is 1-coercive and \(f\) is \(\kernel\)-prox-bounded with \(\lambda<\pb\), then the second inequality too holds as equality.
	\end{enumerate}
\end{fact}

Lower semicontinuity of \(-\env{f}\), while not explicitly stated in \cite[Lem. 3.14]{themelis2026natural}, is a straightforward consequence of continuity of \(\kernel*\circ\nabla\kernel(y)=\innprod{y}{\nabla\kernel(y)}-\kernel(y)\) on \(\Y\) and lower semicontinuity of Fenchel conjugates (or, more generally, of the fact that \(-\env_\lambda{f}\) is the supremum of lsc functions).
Under Legendreness, notice further that \(\env{f}\circ\nabla\kernel*\) is a difference-of-convex functions according to \cref{thm:envconj}.
Under an additional 1-coercivity assumption, this composition enjoys nice regularity properties that we summarize next, and that will play an important role in establishing the bottom-right characterization of \cref{fig:cvx}.
This result parallels the well-known properties of the Moreau envelope in the Euclidean setting as reported in \cite[Ex. 10.32]{rockafellar1998variational}, whose proof is here straightforwardly adaptated to the Bregman setting.
Importantly, composing with \(\nabla\kernel*\) allows us to bypass second-order assumptions required in related results, such as those in \cite{kan2012moreau}.
We also mention that a weaker subdifferential inclusion holding without 1-coercivity will be provided in \cref{thm:lsubdiffenv}.

\begin{lemma}[subsmoothness of the negative envelope]\label{thm:upperC1}%
	Suppose that \(\kernel\) is Legendre and 1-coercive, and let \(\func{f}{\X}{\Rinf}\) be proper, lsc, and \(\pb\)-prox-bounded.
	Then, for any \(\lambda\in(0,\pb)\) the function \(\func{h_\lambda\coloneqq\env{f}\circ\nabla\kernel*}{\R^n}{\Rinf}\) is (finite-valued and) upper-\(C^1\) in the sense of \cite[Def. 10.29]{rockafellar1998variational}, and in particular locally Lipschitz continuous.
	Moreover, for any \(\eta\in\R^n\) one has that
	\begin{align*}
		\rsubdiff[-h_\lambda](\eta)
	=
		\lsubdiff[-h_\lambda](\eta)
	={} &
		\lambda^{-1}
		\bigl[\conv\prox_{\lambda f}-\id\bigr]\circ\nabla\kernel*(\eta),
	\\
		\lsubdiff h_\lambda(\eta)
	\subseteq{} &
		\lambda^{-1}
		\bigl[\id-\prox_{\lambda f}\bigr]\circ\nabla\kernel*(\eta).
	\end{align*}
	Furthermore,  \(\rsubdiff h_\lambda=\lambda^{-1}[\id-\prox_{\lambda f}]\circ\nabla\kernel*\) wherever the set on the right-hand side is a singleton (or, equivalently, wherever \(h_\lambda\) is differentiable), and is empty otherwise.
\end{lemma}
\begin{proof}
	Using the identity \(\kernel(\nabla\kernel*(\eta))-\innprod{\eta}{\nabla\kernel*(\eta)}=-\kernel*(\eta)\) we note that
	\[
		-\lambda h_\lambda(\eta)
	=
		-\inf_{x\in\X}\set{\lambda f(x)+\D(x,\nabla\kernel*(\eta))}
	=
		\sup_{x\in\R^n}\set*{
			\overbracket[0.5pt]{
				\innprod{\eta}{x} - \kernel*(\eta) - (\lambda\tilde f+\kernel)(x)
			}^{\coloneqq f_x(\eta)}
		}.
	\]
	Further noting that the supremum is attained on \(\prox_{\lambda f}\circ\nabla\kernel*(\eta)\), the properties assessed in \cref{thm:mainprop} ensure that any \(\bar\eta\in\R^n\) admits a neighborhood \(\mathcal N_{\bar\eta}\) together with a compact set \(X\subseteq\R^n\) such that
	\[
		-\lambda h_\lambda(\eta)
	=
		\max_{x\in X}f_x(\eta)
	\quad
		\forall\eta\in\mathcal{N}_{\bar\eta}.
	\]
	Since \(f_x(\eta)\) is continuous and its gradient depends continuously on \((x,\eta)\), \(-\lambda h_\lambda\) is lower-\(C^1\) according to \cite[Def. 10.29]{rockafellar1998variational}.
	All the claims then follow from \cite[Thm. 10.31 and Cor. 9.21]{rockafellar1998variational}.
\end{proof}

\begin{fact}[Euclidean representations {\cite[Cor. 3.18]{themelis2026natural}}]\label{thm:Euclidean}%
	Given \(\func{f}{\X}{\Rinf}\), extend it to \(\func{\tilde f}{\R^n}{\Rinf}\) by setting it to \(\infty\) on \(\R^n\setminus\X\) as in \cref{def:fext}.
	Then, for any \(\lambda>0\) the following identities hold:
	\begin{enumerate}
	\item
		\(
			\prox_{\lambda f}
		=
			\Eprox_{\lambda\left(\tilde f+\frac{\kernel-\j}{\lambda}\right)}\circ\nabla\kernel
		\).
	\item
		\(
			\env{f}
		=
			\bigl[
				\Eenv_\lambda(\tilde f+\tfrac{\kernel-\j}{\lambda})
				+
				\tfrac{\kernel*-\j}{\lambda}
			\bigr]
			\circ
			\nabla\kernel
		\).
	\end{enumerate}
	If \(\kernel\) is Legendre and 1-coercive, then the following also holds
	\begin{enumerate}[resume]
	\item
		\(
			\hull_\lambda{f}
		=
			\bigl[
				\Ehull_\lambda(\tilde f+\tfrac{\kernel-\j}{\lambda})
				-
				\tfrac{\kernel-\j}{\lambda}
			\bigr]\restr_{\X}
		\),
		where \(\Ehull_\lambda{f}\coloneqq\hull_\lambda^{\j}f\) denotes the Euclidean proximal hull \cite[Ex. 1.44]{rockafellar1998variational}.
	\end{enumerate}
\end{fact}

	\section{The Bregman level proximal subdifferential}\label{sec:psubdiff}
		We henceforth pattern the analysis of \cite{wang2023every} for the Euclidean case to identify ``proximal'' subdifferentials that enable a resolvent representation of the proximal map.
To begin with, observe that \(\bar x\in\prox_{\lambda f}(\bar y)\) iff
\begin{subequations}\label{subeq:proxchar}%
	\begin{align*}
		f(x)
	\geq{} &
		f(\bar x)
		-
		\tfrac{1}{\lambda}\D(x,\bar y)
		+
		\tfrac{1}{\lambda}\D(\bar x,\bar y)
	\quad
		\forall x\in\X,
	\shortintertext{and as long as \(\bar x\in\interior\X\), with the three-point identity (\cref{thm:3p}) this inequality can be expressed as}
		f(x)
	\geq{} &
		f(\bar x)
		+
		\tfrac{1}{\lambda}
		\innprod{\nabla\kernel(\bar y)-\nabla\kernel(\bar x)}{x-\bar x}
		-
		\tfrac{1}{\lambda}\D(x,\bar x)
	\quad
		\forall x\in\X.
	\end{align*}
	This expression is reminiscent of a relatively weak convexity inequality as in \cite[Def. 3.3]{wang2024mirror}, where the vector \(\bar u\coloneqq\tfrac{1}{\lambda}[\nabla\kernel(\bar y)-\nabla\kernel(\bar x)]\) plays the role of a subgradient of \(f\) at \(\bar x\).
	In fact, it is not difficult to see that all vectors \(\bar u\) for which such global inequality holds form a subset of the Fréchet subdifferential of \(f\) at \(\bar x\); see \cref{thm:FPsubdiff} for the details.
	A similar computation for the right proximal map reveals that \(\bar y\in\prox*_{\lambda g}(\bar x)\) iff
	\[
		g(y)
	\geq
		g(\bar y)
		+
		\tfrac{1}{\lambda}
		\innprod{\bar x-\bar y}{\nabla\kernel(y)-\nabla\kernel(\bar y)}
		-
		\tfrac{1}{\lambda}
		\D(\bar y,y)
	\quad
		\forall y\in\Y,
	\]
	leading to the following definitions.
\end{subequations}

\begin{definition}[Bregman level proximal subdifferentials]\label{def:psubdiff}%
	The \emph{left (Bregman) \(\lambda\)-level proximal subdifferential} of a function \(\func{f}{\X}{\Rinf}\) is \(\ffunc{\psubdiff f}{\X}{\R^n}\), where \(\bar u\in\psubdiff f(\bar x)\) if
	\begin{equation}\label{eq:psubdiff}
		\bar x\in\interior\X
	\quad\text{and}\quad
		f(x)
		\geq
		f(\bar x)
		+
		\innprod{\bar u}{x-\bar x}
		-
		\tfrac{1}{\lambda}\D(x,\bar x)
	\quad
		\forall x\in\X.
	\end{equation}
	Similarly, the \emph{right (Bregman) \(\lambda\)-level proximal subdifferential} of \(\func{g}{\Y}{\Rinf}\) is \(\ffunc{\psubdiff*g}{\Y}{\R^n}\), where \(\bar v\in\psubdiff*g(\bar y)\) if
	\begin{equation}\label{eq:psubdiff*}
		g(y)
	\geq
		g(\bar y)
		+
		\innprod{\bar v}{\nabla\kernel(y)-\nabla\kernel(\bar y)}
		-
		\tfrac{1}{\lambda}
		\D(\bar y,y)
	\quad
		\forall y\in\Y.
	\end{equation}
\end{definition}

When \(\kernel=\j\), both \(\psubdiff f\) and \(\psubdiff* g\) reduce to the level proximal subdifferentials \(\Epsubdiff f\) and \(\Epsubdiff g\) originally introduced by Rockafellar in \cite[eq. (2.13)]{rockafellar2021characterizing} and further investigated by Wang and Wang \cite{wang2023every} and by Luo et al. \cite{luo2024level}.
Acute readers may notice that \(\psubdiff*g\), although referred to as a (level proximal) subdifferential of \(g\), does not effectively represent a subdifferential of \(g\).
However, as we shall soon see, \(\psubdiff*g\) is in fact a subdifferential of \(g\circ\nabla\kernel*\) when \(\kernel\) is Legendre; see \cref{thm:FPsubdiff:g}.
Although slightly abusive, we stick to this terminology to highlight the parallel with the \emph{left} counterpart and its role in characterizing the proximal map, as well as to keep consistency with the Euclidean case.

\subsection{Resolvent representation}
In this section, we investigate, among other things, the relations between Bregman level proximal subdifferentials and  Bregman proximal operators, which ultimately reveal that Bregman proximal operators are always (warped) resolvents of Bregman level proximal subdifferentials under a standard range assumption.
For a set-valued operator \(\ffunc{T}{X}{Y}\) and \(U\subseteq X\), we adopt a slight abuse of notation in interpreting the restriction of \(T\) on \(U\) as the operator \(\ffunc{T\restr_U}{X}{Y}\) (as opposed to \(\ffunc{T\restr_U}{U}{Y}\)) defined by
\[
	T\restr_U(u)
=
	\begin{cases}
		T(u) & \text{if }u\in U
	\\
		\emptyset & \text{otherwise.}
	\end{cases}
\]

\begin{lemma}
	For any set-valued operator \(\ffunc{T}{X}{Y}\) between two nonempty sets \(X\) and \(Y\), the following hold:
	\begin{enumerate}
	\item \label{thm:TU-1}%
		For any \(U\subseteq X\), one has that \(U\cap T^{-1}=T\restr_U^{-1}\).
	\item \label{thm:T-1V}%
		For any \(V\subseteq Y\), one has that \((V\cap T)^{-1}=T^{-1}\restr_V\).
	\end{enumerate}
\end{lemma}
\begin{proof}
	In what follows, let \(\bar u\in U\) and \(\bar v\in V\) be fixed.
	\begin{itemize}
	\item ``\ref{thm:TU-1}''
		\(
			\bar u\in U\cap T^{-1}(\bar v)
		\Leftrightarrow
			\bar u\in U\) and \(\bar v\in T(\bar u)
		\Leftrightarrow
			\bar v\in T\restr_U(\bar u)
		\Leftrightarrow
			\bar u\in T\restr_U^{-1}(\bar v)
		\).

	\item ``\ref{thm:T-1V}''
		\(
			\bar u\in (V\cap T)^{-1}(\bar v)
		\Leftrightarrow
			\bar v\in V\) and \(\bar u\in T^{-1}(\bar v)
		\Leftrightarrow
			\bar u\in T^{-1}\restr_V(\bar v)
		\).
	\qedhere
	\end{itemize}
\end{proof}

\begin{fact}[{\cite[Lem. 3.3 and Thm. 3.6]{benoist1996what}}]\label{thm:subdiffconv}%
	Let \(\func{h}{\R^n}{\Rinf}\) be proper, lsc, and 1-coercive.
	Then, for every \(x\in\conv\dom h\) there exists a collection of points \(x_i\in\dom h\) and \(\alpha_i>0\), \(i\in I\), with \(\sum_{i\in I}\alpha_i=1\) such that
	\[
		x=\sum_{i\in I}\alpha_i x_i,
	\quad
		\conv h(x)=\sum_{i\in I}\alpha_ih(x_i),
	\quad\text{and}\quad
		\lsubdiff (\conv h)(x)=\bigcap_{i\in I}\fsubdiff h(x_i).
	\]
\end{fact}

Below is the first main result of this section.

\begin{theorem}\label{thm:FPsubdiff}%
	For any \(\lambda>0\), \(\func{f}{\X}{\Rinf}\) and \(\func{g}{\Y}{\Rinf}\) the following hold:
	\begin{enumerate}
	\item \label{thm:FPsubdiff:f}%
		\(
			\psubdiff f
		=
			\bigl[\fsubdiff(f+\lambda^{-1}\kernel)\bigr]\restr_{\interior\X}-\lambda^{-1}\nabla\kernel
		\subseteq
			\rsubdiff f
		\).\footnote{%
			The inclusion \(\psubdiff f\subseteq\rsubdiff f\) is meant in the sense that \(\graph\psubdiff f\subseteq\graph\rsubdiff f\), or, equivalently, that \(\psubdiff f(x)\subseteq\rsubdiff f(x)\) holds for any \(x\in\interior\X\).
			Similar conventions are adopted throughout.
		}%
	\item \label{thm:Jfsubdiff}%
		\(
			\prox_{\lambda f}
		=
			\fsubdiff(\kernel+\lambda f)^{-1}\circ\nabla\kernel
		\).
	\item \label{thm:prox-J}%
		\(
			\interior\X\cap\prox_{\lambda f}
		=
			\bigl(\nabla\kernel+\lambda\psubdiff f\bigr)^{-1}\circ\nabla\kernel
		\).
	\item \label{thm:prox*-J}%
		\(
			\prox*_{\lambda g}=
			\bigl(\id+\lambda\psubdiff*g\bigr)^{-1}\restr_\X
		\).
	\end{enumerate}
	If \(\kernel\) is Legendre, then the following also hold:
	\begin{enumerate}[resume*]
	\item \label{thm:FPsubdiff:g}%
		\(
			\psubdiff*g
		=
			\fsubdiff(g\circ\nabla\kernel*+\lambda^{-1}\kernel*)-\lambda^{-1}\nabla\kernel*
		\subseteq
			\rsubdiff[g\circ\nabla\kernel*]
		\).
	\item \label{thm:Jfsubdiff*}%
		\(
			\prox*_{\lambda g}
		=
			\fsubdiff(\kernel*+\lambda g\circ\nabla\kernel*)^{-1}
		\).
	\item \label{thm:monotone}%
		\(
			\prox_{\lambda f}\circ\nabla\kernel*
		=
			\fsubdiff(\kernel+\lambda f)^{-1}\restr_{\Y*}
		\)
		and \(\prox*_{\lambda g}\) are monotone operators.
	\item \label{thm:lsubdiffenv}%
		\(
			\prox_{\lambda f}\circ\nabla\kernel*
		\subseteq
			\lsubdiff(\kernel+\lambda\tilde f)^*\restr_{\Y*}
		=
			\nabla\kernel*
			+
			\lambda\lsubdiff[-\env{f}\circ\nabla\kernel*]
		\).
		When \(\kernel+\tilde\lambda f\) is lsc and convex, the inclusion holds as equality.
	\end{enumerate}
	If, additionally, \(\kernel\) is 1-coercive, then the following also holds:
	\begin{enumerate}[resume]
	\item \label{thm:proxhull}%
		\(
			\conv\prox_{\lambda f}\circ\nabla\kernel*
		=
			\prox_{\lambda \hull_\lambda{f}}\circ\nabla\kernel*
		\).
	\end{enumerate}
\end{theorem}
\begin{proof}
	We start with the first four claims without extra assumptions on \(\kernel\).
	\begin{itemize}
	\item ``\ref{thm:FPsubdiff:f}''
		The claimed identity holds trivially outside \(\interior\X\).
		Let \(\bar x\in\interior\X\) be fixed.
		After expanding \(\D(x,\bar x)\), the inequality in \eqref{eq:psubdiff} can be cast as
		\[
			[f+\lambda^{-1}\kernel](x)
			\geq
			[f+\lambda^{-1}\kernel](\bar x)
			+
			\innprod{\bar u+\lambda^{-1}\nabla\kernel(\bar x)}{x-\bar x}
		\quad
			\forall x\in\X.
		\]
		This shows that \(\bar u\in\psubdiff f(\bar x)\) iff \(\bar u+\lambda^{-1}\nabla\kernel(\bar x)\in\fsubdiff[f+\lambda^{-1}\kernel](\bar x)\), which after adding \(\nabla\kernel(\bar x)\) (well defined, since \(\bar x\in\interior\X\)) on both sides leads to the claimed identity.
		The inclusion \(\psubdiff f(\bar x)\subseteq\rsubdiff f(\bar x)\) is trivial if \(\bar x\notin\interior\X\), since the left-hand side is empty in this case.
		When \(\bar x\in\interior\X\), it follows from \eqref{eq:psubdiff} together with the fact that
		\(
			\D(x,\bar x)=o\bigl(\norm{x-\bar x}\bigr)
		\)
		as \(x\to\bar x\).
		\item ``\ref{thm:Jfsubdiff}''
		Noticing that
		\(
			\prox_{\lambda f}(\bar y)
		=
			\argmin\set{
				\lambda f+\kernel-\innprod{\nabla\kernel(\bar y)}{{}\cdot{}}
			}
		\),
		the claim follows from the fact that
		\(
			\argmin\set{\psi-\innprod{u}{{}\cdot{}}}
		=
			\bigl[\fsubdiff \psi\bigr]^{-1}(u)
		\)
		holds for any \(\func{\psi}{\X}{\Rinf}\) and \(u\in\R^n\).

	\item ``\ref{thm:prox-J}''
		Assertion \ref{thm:Jfsubdiff} in combination with \cref{thm:TU-1} implies that
		\(
			\interior\X\cap\prox_{\lambda f}
		=
			\bigl(\bigl[\fsubdiff(\kernel+\lambda f)\bigr]\restr_{\interior\X}\bigr)^{-1}\circ\nabla\kernel
		\).
		In turn, the identity in assertion \ref{thm:FPsubdiff:f} yields the claim.

	\item ``\ref{thm:prox*-J}''
		For any \(\bar x\in\X\) we have \(\bar y\in\prox*_{\lambda g}(\bar x)\) iff
		\begin{align*}
			\forall y\in\Y
		\quad
			g(y)
		\geq{} &
			g(\bar y)
			+
			\tfrac{1}{\lambda}
			\D(\bar x,\bar y)
			-
			\tfrac{1}{\lambda}
			\D(\bar x,y),
		\shortintertext{%
			which using the three-point identity of \cref{thm:3p} expands to
		}
			\forall y\in\Y
		\quad
			g(y)
		\geq{} &
			g(\bar y)
			-
			\tfrac{1}{\lambda}
			\innprod{\bar x-\bar y}{\nabla\kernel(\bar y)-\nabla\kernel(y)}
			-
			\tfrac{1}{\lambda}
			\D(\bar y,y).
		\end{align*}
		This shows that
		\begin{align*}
			\bar y\in\prox*_{\lambda g}(\bar x)
		\Leftrightarrow{} &
			\bar x\in\X
			\text{ and }
			\tfrac{\bar x-\bar y}{\lambda}\in\psubdiff*g(\bar y)
		\\
		\Leftrightarrow{} &
			\bar x\in\X
			\text{ and }
			\bar y\in(\id+\lambda\psubdiff*g)^{-1}(\bar x).
		\end{align*}
		Invoking \cref{thm:T-1V} results in the claimed expression.
	\end{itemize}

	\noindent
	In what follows, suppose that \(\kernel\) is Legendre.
	\begin{itemize}
	\item ``\ref{thm:FPsubdiff:g}''
		By appealing to \cref{thm:D*} and denoting \(\eta=\nabla\kernel(y)\) and \(\bar\eta=\nabla\kernel(\bar y)\), the inequality in \eqref{eq:psubdiff*} reads
		\[
			g\circ\nabla\kernel*(\eta)
		\geq
			g\circ\nabla\kernel*(\bar\eta)
			+
			\innprod{\bar v}{\eta-\bar\eta}
			-
			\tfrac{1}{\lambda}
			\D*(\eta, \bar\eta)
		\quad
			\forall\eta\in\Y*,
		\]
		and the same conclusion follows.

	\item ``\ref{thm:Jfsubdiff*}''
		Similarly as in the proof of assertion \ref{thm:Jfsubdiff}, the claim follows by observing that
		\begin{align*}
			\prox*_{\lambda g}(\bar x)
		={} &
			\argmin_{y\in\Y}\set{
				\lambda g(y)-\kernel(y)-\innprod{\nabla\kernel(y)}{\bar x-y}
			}
		\\
		={} &
			\argmin_{y\in\Y}\set{
				\lambda g(y)+\kernel*(\nabla\kernel(y))-\innprod{\nabla\kernel(y)}{\bar x}
			}
		\\
		={} &
			\argmin_{\eta\in\Y*}\set{
				\lambda g\circ\nabla\kernel*(\eta)+\kernel*(\eta)-\innprod{\eta}{\bar x}
			}.
		\end{align*}

	\item ``\ref{thm:monotone}''
		The identity follows from assertion \ref{thm:Jfsubdiff}.
		Combined with the expression in \ref{thm:Jfsubdiff*}, the claims on monotonicity owe to the fact that \(\fsubdiff h\) is a monotone operator for any function \(h\), and as such so is its inverse.

	\item ``\ref{thm:lsubdiffenv}''
		For any \(\eta\in\Y*\), it follows from assertion \ref{thm:monotone} together with \cref{fact:fsubdiff} that
		having \(x\in\prox_{\lambda f}\circ\nabla\kernel*(\eta)\) amounts to \(\fsubdiff(\lambda\tilde f+\kernel)(x)\neq\emptyset\) and \(\eta\in\fsubdiff\conv*(\lambda\tilde f+\kernel)(x)\), this latter inclusion being equivalent to \(x\in\lsubdiff(\lambda\tilde f+\kernel)^*(\eta)\).
		In turn, the subdifferential inclusion follows from \cref{thm:envconj} and the calculus rule of \cite[Ex. 8.10]{rockafellar1998variational}.
		When \(\lambda\tilde f+\kernel\) is lsc and convex, one has that
		\[
			\fsubdiff(\lambda f+\kernel)^{-1}
		=
			\fsubdiff(\lambda\tilde f+\kernel)^{-1}
		=
			\lsubdiff(\lambda\tilde f+\kernel)^*
		\]
		and the claimed identity thus follows from assertion \ref{thm:monotone}.
	\end{itemize}

	\noindent
	To conclude, suppose that \(\kernel\) is both Legendre and 1-coercive.
	\begin{itemize}
	\item ``\ref{thm:proxhull}''
		In order to simplify the notation, let \(h\coloneqq\hull_\lambda{f}\).
		It follows from \cref{thm:hull} that \(\lambda h+\kernel=\conv*(\lambda f+\kernel)\) is convex, to which \cite[Prop. 2.10(ii)]{wang2022bregman} entails
		\begin{equation}\label{eq:proxhull}
			\prox_{\lambda h}\circ\nabla\kernel*
		=
			\bigl[\lsubdiff(\lambda h+\kernel)\bigr]^{-1}
		=
			\bigl[\lsubdiff\conv*(\lambda f+\kernel)\bigr]^{-1}.
		\end{equation}
		We proceed to prove the claimed identity:

		\begin{itemize}
		\item ``\(\subseteq\)''
			Let \(\bar x\in\prox_{\lambda f}\circ\nabla\kernel*(\bar\eta)\) be fixed, so that
			\(
				\bar\eta
			\in
				\fsubdiff(\lambda f+\kernel)(\bar x)
			=
				\lsubdiff\conv*(\lambda f+\kernel)(\bar x)
			\),
			where the inclusion follows from assertion \ref{thm:Jfsubdiff} and the identity, in turn, from \cref{fact:fsubdiff}.
			Equivalently,
			\(
				\bar x
			\in
				\bigl[\lsubdiff\conv*(\lambda f+\kernel)\bigr]^{-1}(\bar\eta)
			=
				\prox_{\lambda h}\circ\nabla\kernel*(\bar\eta)
			\)
			by \eqref{eq:proxhull}.
			From the arbitrariness of \(\bar x\) and \(\bar\eta\) we conclude that
			\[
				\conv\prox_{\lambda f}\circ\nabla\kernel*(\bar\eta)
			\subseteq
				\conv\prox_{\lambda h}\circ\nabla\kernel*(\bar\eta).
			\]
			Since \(\lambda h+\kernel\) is convex, \cite[Prop. 2.10(ii)]{wang2022bregman} yields that \(\prox_{\lambda h}\circ\nabla\kernel*(\bar\eta)\) is convex, hence that the convex hull in the right-hand side of the inclusion above is superfluous.

		\item ``\(\supseteq\)''
			Consider now \(\bar x\in\prox_{\lambda h}\circ\nabla\kernel*(\bar\eta)\), so that \(\bar\eta\in\lsubdiff\conv*(\lambda f+\kernel)(\bar x)\) by \eqref{eq:proxhull}.
			Notice that \(\lambda\tilde f+\kernel\) too is 1-coercive by \cref{thm:extlsc:lsc}, where \(\tilde f\) is as in \cref{def:fext}, so that \cref{thm:subdiffconv} ensures that \(\bar x\) is a convex combination of points \(\bar x_i\), \(i\in I\), with
			\[
				\bar\eta
				\in
				\bigcap_{i\in I}\fsubdiff(\lambda\tilde f+\kernel)(\bar x_i)
			~~\Leftrightarrow~~
				\bar x_i\in \bigl[\oldpartial_F(\lambda f+\kernel)\bigr]^{-1}(\bar\eta)
				=
				\prox_{\lambda f}\circ\nabla\kernel*(\bar\eta)
				~~
				\forall i\in I.
			\]
			This shows that \(\bar x\in\conv \prox_{\lambda f}\circ\nabla\kernel*(\bar\eta)\), and the sought inclusion follows from the arbitrariness of \(\bar x\) and \(\bar\eta\).
		\qedhere
		\end{itemize}
	\end{itemize}
\end{proof}

\Cref{thm:prox-J} suggests that the following assumption should be in place in order for the intersections with \(\interior\X\) in the formulas therein be superfluous.
\begin{assumption}\label{ass:range}%
	\(\range\prox_{\lambda f}\subseteq\interior\dom\kernel\).
\end{assumption}

\begin{remark}\label{rem:range}%
	This assumption is ubiquitous for ensuring well-defined Bregman minimization algorithms, as it allows the output of one proximal operation to be used as input for the next iteration.
	We mention some simple known sufficient conditions ensuring its validity:
	\begin{enumerate}
	\item
		When \(f\) is convex and \(\dom f\cap\interior\X\neq\emptyset\), this requirement is automatically guaranteed by standard subdifferential calculus arguments.

	\item
		More generally, under essential smoothness of \(\kernel\) a sufficient condition ensuring \cref{ass:range} is the following constraint qualification on the boundary:
		\[
			\hsubdiff f(x)\cap\bigl(-\hsubdiff\kernel(x)\bigr)=\set{0}
		\quad\forall x\in\dom f\cap\boundary\X,
		\]
		where \(\hsubdiff\) denotes the \emph{horizon subdifferential} \cite[Def. 8.3(c)]{rockafellar1998variational}.
		Indeed, by \cite[Cor. 10.9]{rockafellar1998variational} this condition enables the calculus rule \(\lsubdiff(\kernel+\lambda f)(x)\subseteq\lsubdiff\kernel(x)+\lambda\lsubdiff f(x)=\emptyset\) at any boundary point \(x\), this set being empty by essential smoothness, thus preventing \(x\) from being a minimizer in the proximal minimization subproblem \eqref{eq:prox}.

	\item
			Still under essential smoothness of \(\kernel\), a simpler geometrical condition that only involves \(f\) is the existence, for any \(x\in\dom f\cap\boundary\X\), of a \emph{strictly feasible direction}\footnote{%
				With strictly feasible direction at \(x\) we mean any \(v\in\R^n\) such that \(x+\tau v\in\interior\X\) for \(\tau>0\) small enough.
				By convexity, up to positive scalings these are precisely those of the form \(v=\frac{x'-x}{\norm{x'-x}}\) with \(x'\in\interior\X\).
			}
			\(v\in\R^n\) along which the directional derivative (either does not exist or) is not \(+\infty\).
			Indeed, this guarantees through \cite[Lem. 26.2]{rockafellar1970convex} that \(\kernel+\lambda f\) has a downward vertical slope at \(x\) (more precisely, a negative infinite lower Dini derivative) along \(v\), which forces \(\rsubdiff(\kernel+\lambda f)(x)\) to be empty.
	\end{enumerate}
\end{remark}

Rather than algorithmic implications, our interest lies in the anticipated \emph{full resolvent characterization} that emerges from the assumption.
This characterization is expressed in terms of the newly introduced \emph{level proximal subdifferentials}, which generalize the role played by the Fenchel subdifferential for convex functions.
\begin{corollary}\label{cor:resolvent}
	For any \(\lambda>0\), \(\func{f}{\X}{\Rinf}\) and \(\func{g}{\Y}{\Rinf}\) the following hold:
	\begin{enumerate}
	\item \label{thm:warped}%
		\emph{Warped resolvent representation of \(\prox_{\lambda f}\):}
		if \cref{ass:range} holds, then
		\begin{equation}
			\prox_{\lambda f}=(\nabla\kernel+\lambda\psubdiff f)^{-1}\circ\nabla\kernel.
		\end{equation}

	\item
		\emph{Resolvent representation of \(\prox*_{\lambda g}\):}
		if \(\X=\R^n\), then
		\begin{equation}
			\prox*_{\lambda g}=(\id+\lambda\psubdiff*g)^{-1}.
		\end{equation}
	\end{enumerate}
\end{corollary}

We next show that left and right level proximal subdifferentials are mutually related in a similar fashion as the proximal mappings when \(\kernel\) is of Legendre type.

\begin{lemma}
	Let \(\kernel\) be of Legendre type.
	Then, for any \(\func{g}{\Y}{\Rinf}\) one has
	\begin{align*}
		\psubdiff*g
	={} &
		\psubdiff^{\kernel*}G\circ\nabla\kernel,
	\shortintertext{%
		where \(\func{G}{\X*}{\Rinf}\) equals \(g\circ\nabla\kernel*\) on \(\interior\X*\) and \(\infty\) elsewhere,\footnotemark{} while for \(\func{f}{\X}{\Rinf}\) one has
	}
		\psubdiff f
	\subseteq{} &
		\psubdiff*^{\kernel*}F\circ\nabla\kernel,
	\end{align*}
	\footnotetext{%
		Defining \(G\) in this way, rather than simply using \(g\circ\nabla\kernel*\), is a minor technicality.
		It arises from the fact that the left-level proximal subdifferential is defined for functions on the \emph{domain} of the dgf, rather than merely its \emph{interior}.%
	}%
	where \(\func{F\coloneqq f\circ\nabla\kernel*}{\Y*}{\Rinf}\).
	A sufficient condition for the inclusion to hold everywhere as equality is continuity of \(\kernel\) on its domain together with the existence, for any \(x\in\X\setminus\interior\X\), of a sequence \(\interior\X\ni x^k\to x\) such that
	\begin{equation}\label{eq:subdiff*:CQ}
		f(x)
	=
		\lim_{k\to\infty}f(x^k).
	\end{equation}
\end{lemma}
\begin{proof}
	We begin by proving the claim for \(f\).
	We have
	\begin{align*}
		\bar u\in\psubdiff f(\bar x)
	\quad\defeq*[\Rightarrow]\quad &
			f(x)
		\geq
			f(\bar x)
			+
			\innprod{\bar u}{x-\bar x}
			-
			\tfrac{1}{\lambda}\D(x,\bar x)
		\quad
			\forall x\in\interior\X
	\\
	&
		\text{(and \(\bar x\in\interior\X\))}
	\\
	\quad\Leftrightarrow\quad &
			F(\xi)
		\geq
			F(\bar\xi)
			+
			\innprod{\bar u}{\nabla\kernel*(\xi)-\nabla\kernel*(\bar\xi)}
			-
			\tfrac{1}{\lambda}\D*(\bar\xi,\xi)
		~~
			\forall\xi\in\Y*
	\\
	\defeq*[\Leftrightarrow]\quad &
			\bar u\in\psubdiff*^{\kernel*}F(\bar\xi),
	\end{align*}
	where the first equivalence follows from \cref{thm:D*} with \(\xi\coloneqq\nabla\kernel(x)\), \(\eta\coloneqq\nabla\kernel(y)\), and \(\bar\xi\) and \(\bar\eta\) defined similarly.
	Under continuity of \(\kernel\) on its domain and the condition in \eqref{eq:subdiff*:CQ}, the first inequality above can be extended to all \(x\in\X\) (so that the implication therein can be reversed).
	Indeed, for any \(x\in\X\setminus\interior\X\) there exists a sequence \(\interior\X\ni x^k\to x\) such that \((f-\lambda^{-1}\kernel)(x^k)\to(f-\lambda^{-1}\kernel)(x)\), and a limiting argument in the inequality yields the claim.

	We now turn to the claim for \(g\).
	Adopting the same notation for \(\xi,\eta,\bar\xi,\bar\eta\) as above, we have
	\begin{align*}
		\bar v\in\psubdiff*g(\bar y)
	\quad\defeq*[\Leftrightarrow]\quad &
			g(y)
		\geq
			g(\bar y)
			+
			\innprod{\bar v}{\nabla\kernel(y)-\nabla\kernel(\bar y)}
			-
			\tfrac{1}{\lambda}\D(\bar y,y)
		\quad
			\forall y\in\Y
	\\
	\quad\Leftrightarrow\quad &
			G(\eta)
		\geq
			G(\bar\eta)
			+
			\innprod{\bar v}{\eta-\bar\eta}
			-
			\tfrac{1}{\lambda}\D*(\eta,\bar\eta)
		\quad
			\forall\eta\in\interior\X*
	\\
	\quad\Leftrightarrow\quad &
			G(\eta)
		\geq
			G(\bar\eta)
			+
			\innprod{\bar v}{\eta-\bar\eta}
			-
			\tfrac{1}{\lambda}\D*(\eta,\bar\eta)
		\quad
			\forall\eta\in\X*,
	\end{align*}
	where the second equivalence follows from \cref{thm:D*}, and the last one from the fact that \(\dom G\subseteq\interior\X*\).
	By definition, this means that \(\bar v\in\psubdiff^{\kernel*}G(\bar\eta)\), yielding the claim.
\end{proof}

\begin{remark}\label{rem:CQ:mathringf}%
	When \(f\) is lsc, the boundary qualification \eqref{eq:subdiff*:CQ} indicates that \(f\) is fully characterized by its value onto \(\interior\X\).
	Indeed, this condition can be expressed as \(f=\closure\mathring f\), where \(\func{\mathring f}{\X}{\Rinf}\) is given by \(\mathring f(x)=f(x)\) for \(x\in\interior\X\), and \(\infty\) otherwise.
\end{remark}

\subsection{Pointwise characterizations}
The goal of this section is to characterize the Bregman level proximal subdifferential, beginning with a result concerning the fundamental question: \textit{when does \(\psubdiff f\) exist?}
The definition \eqref{eq:psubdiff} is always an option, but it may be difficult to verify.
We now present equivalent descriptions for \(\psubdiff f\).
Akin to classic subdifferentials, which are often characterized by the behavior of minorants with more favorable properties,
the minorant \(\hull_\lambda{f}\) determines \(\psubdiff f\) completely.

\begin{theorem}\label{thm:psubdiffchar}%
	Suppose that \(\kernel\) is 1-coercive and Legendre, and let \(\func{f}{\X}{\Rinf}\) be proper, lsc, and \(\kernel\)-prox-bounded.
	Then, for any \(\lambda\in(0,\pb)\), \(\bar x\in\interior\X\), \(\bar u\in\R^n\), and denoting \(\bar y\coloneqq\nabla\kernel^*(\lambda\bar u+\nabla \kernel(\bar x))\) the following are equivalent:
	\begin{enumerateq}
	\item \label{thm:psubdiffchar:nonempty}%
		\(\bar u\in \psubdiff f(\bar x)\).
	\item \label{thm:psubdiffchar:prox}%
		\(\bar x\in \prox_{\lambda f}(\bar y)\).
	\item
		\label{thm:psubdiffchar:env}%
		\(\frac{\bar x-\bar y}{\lambda}\in\psubdiff* (-\env{f})(\bar y)\) and \(\bar x\in\dom\psubdiff f\).
	\item \label{thm:psubdiffchar:prox*}%
		\(\bar y\in \prox*_{-\lambda\env{f}}(\bar x)\) and \(\bar x\in\dom\psubdiff f\).
	\item \label{thm:psubdiffchar:proxhull}%
		\(\bar x\in \prox_{\lambda\hull_\lambda{f}}(\bar y)\) and \(\hull_\lambda{f}(\bar x)=f(\bar x)\).
	\item
		\label{thm:psubdiffchar:hull}%
		\(\bar u\in\lsubdiff\hull_\lambda{f}(\bar x)\)
		and
		\(\hull_\lambda{f}(\bar x)=f(\bar x)\).
	\item
		\label{thm:psubdiffchar:chull}%
		\(\nabla \kernel(\bar y)\in\lsubdiff\conv(\lambda f+\kernel)(\bar x)\)
		and
		\(\conv(\lambda f+\kernel)(\bar x)=(\lambda f+\kernel)(\bar x)\).
	\end{enumerateq}
\end{theorem}

\begin{proof}
	We start by observing that the definition of \(\bar y\) can equivalently be expressed as
	\begin{equation}\label{eq:yu}
		\bar u
	=
		\tfrac{\nabla\kernel(\bar y)-\nabla\kernel(\bar x)}{\lambda}.
	\end{equation}
	\begin{itemize}
	\item
		``\ref{thm:psubdiffchar:nonempty} \(\Leftrightarrow\) \ref{thm:psubdiffchar:prox}''
		Taking into account \eqref{eq:yu}, this is \cref{thm:prox-J}.

	\item
		``\ref{thm:psubdiffchar:env} \(\Leftrightarrow\) \ref{thm:psubdiffchar:prox*}''
		This is \cref{thm:prox*-J}.

	\item ``\ref{thm:psubdiffchar:proxhull} \(\Rightarrow\) \ref{thm:psubdiffchar:hull}''
		Since \(\psubdiff(\hull_\lambda{f})\subseteq\lsubdiff(\hull_\lambda{f})\) by \cref{thm:FPsubdiff:f}, the claim follows from the implication ``\ref{thm:psubdiffchar:prox} \(\Rightarrow\) \ref{thm:psubdiffchar:nonempty}'' shown above applied to \(\hull_\lambda{f}\).

	\item ``\ref{thm:psubdiffchar:hull} \(\Rightarrow\) \ref{thm:psubdiffchar:proxhull}''
		Note that
		\begin{align*}
			\lsubdiff\hull_\lambda{f}(\bar x)
		~~\eq*[\ref{thm:hull}]~~ &
			\lsubdiff\conv(f+\lambda^{-1}\kernel)(\bar x)-\lambda^{-1}\nabla\kernel(\bar x)
		\\
		~~=~~ &
			\fsubdiff\left(\hull_\lambda{f}+\lambda^{-1}\kernel\right)(\bar x)
			-
			\lambda^{-1}\nabla\kernel(\bar x)
		\\
		~~\eq*[\ref{thm:FPsubdiff:f}]~~ &
			\psubdiff\hull_\lambda{f}(\bar x).
		\end{align*}
		Therefore, by virtue of the implication ``\ref{thm:psubdiffchar:nonempty} \(\Rightarrow\) \ref{thm:psubdiffchar:prox}'' shown above applied to \(\hull_\lambda{f}\), it follows from the inclusion \(\bar u\in\lsubdiff\hull_\lambda{f}(\bar x)\) that \(\bar y\in\prox_{\lambda\hull_\lambda{f}}(\bar x)\).

	\item
		``\ref{thm:psubdiffchar:hull} \(\Leftrightarrow\) \ref{thm:psubdiffchar:chull}''
		Apply \cref{thm:hull}.

	\item
		``\ref{thm:psubdiffchar:nonempty} \(\Rightarrow\) \ref{thm:psubdiffchar:hull}''
		By virtue of the equivalence ``\ref{thm:psubdiffchar:nonempty} \(\Leftrightarrow\) \ref{thm:psubdiffchar:prox}'' proven above,
		\(
			\bar x\in\prox_{\lambda f}(\bar y)
		\).
		Then,
		\begin{align*}
			f(\bar x)
		\geq
			\hull_\lambda{f}(\bar x)
		={} &
			\sup_{w\in\Y}
			\set{
				\env{f}(w)
				-
				\tfrac{1}{\lambda}\D(\bar x,w)
			}
		\\
		\geq{} &
			\env{f}(\bar y)
			-
			\tfrac{1}{\lambda}\D(\bar x,\bar y)
		=
			f(\bar x),
		\end{align*}
		where the first inequality owes to \cref{thm:hull}.
		This shows that \(f(\bar x)=\hull_\lambda{f}(\bar x)\).
		By adding \(\lambda^{-1}\kernel\) to both sides in \eqref{eq:psubdiff}, from the inclusion \(\bar u\in\psubdiff f(\bar x)\) we have that
		\begin{align*}
			f
			+
			\tfrac{1}{\lambda}\kernel
		\geq{} &
			f(\bar x)
			+
			\innprod{\bar u}{{}\cdot{}-\bar x}
			-
			\tfrac{1}{\lambda}\D({}\cdot{},\bar x)
			+
			\tfrac{1}{\lambda}\kernel
		\\
		={} &
			f(\bar x)
			+
			\tfrac{1}{\lambda}\kernel(\bar x)
			+
			\tfrac{1}{\lambda}\innprod{\nabla\kernel(\bar y)}{{}\cdot{}-\bar x}
		\\
		={} &
			\hull_\lambda{f}(\bar x)
			+
			\tfrac{1}{\lambda}\kernel(\bar x)
			+
			\tfrac{1}{\lambda}\innprod{\nabla\kernel(\bar y)}{{}\cdot{}-\bar x}.
		\end{align*}
		By taking the closed convex hull and by invoking \cref{thm:hull} this results in
		\[
			\hull_\lambda{f}+\tfrac{1}{\lambda}\kernel
		\geq
			\hull_\lambda{f}(\bar x)
			+
			\tfrac{1}{\lambda}\kernel(\bar x)
			+
			\tfrac{1}{\lambda}\innprod{\nabla\kernel(\bar y)}{{}\cdot{}-\bar x},
		\]
		that is,
		\[
			\tfrac{1}{\lambda}\nabla\kernel(\bar y)
		\in
			\fsubdiff\bigl[\hull_\lambda{f}+\tfrac{1}{\lambda}\kernel\bigr](\bar x)
		=
			\lsubdiff\bigl[\hull_\lambda{f}+\tfrac{1}{\lambda}\kernel\bigr](\bar x)
		=
			\lsubdiff\hull_\lambda{f}(\bar x)+\tfrac{1}{\lambda}\nabla\kernel(\bar x),
		\]
		where the first equality follows from the convexity of \(\hull_\lambda{f}+\tfrac{1}{\lambda}\kernel\).
		Recalling that \(\lambda^{-1}\nabla\kernel(\bar y)=\bar u+\lambda^{-1}\nabla\kernel(\bar x)\), cf. \eqref{eq:yu}, the proof follows.

	\item
		``\ref{thm:psubdiffchar:chull} \(\Rightarrow\) \ref{thm:psubdiffchar:nonempty}''
		We have,
		\begin{align*}
			(\forall x\in\X)~
			(f+\lambda^{-1}\kernel)(x)
		\geq{} &
			\conv(f+\lambda^{-1}\kernel)(x)
		\\
		\geq{} &
			\conv(f+\lambda^{-1}\kernel)(\bar x)
			+
			\innprod{\bar u+\lambda^{-1}\nabla \kernel(\bar x)}{x-\bar x}
		\\
		={} &
			(f+\lambda^{-1}\kernel)(\bar x)
			+
			\innprod{\bar u+\lambda^{-1}\nabla \kernel(\bar x)}{x-\bar x}
		\\
		={} &
			f(\bar x)
			+
			\innprod{\bar u}{x-\bar x}
			-
			\tfrac{1}{\lambda}\D(x,\bar x)
			+
			\tfrac{1}{\lambda}\kernel(x),
		\end{align*}
		which by \eqref{eq:psubdiff} shows that \(\bar u\in\psubdiff f(\bar x)\).

	\item
		``\ref{thm:psubdiffchar:prox*} \(\Rightarrow\) \ref{thm:psubdiffchar:prox}''
		We have
		\begin{equation}\label{eq:envfxy}
			\hull_\lambda{f}(\bar x)
		\defeq
			-\env*{\bigl(-\env{f}\bigr)}(\bar x)
		=
			\env{f}(\bar y)
			-
			\tfrac{1}{\lambda}\D(\bar x,\bar y).
		\end{equation}
		On the other hand, by \cref{thm:prox-J} there exists \(\bar y'\in\Y\) such that
		\(
			\bar x\in\prox_{\lambda f}(\bar y')
		\)
		(specifically, of the form \(\bar y'=\nabla\kernel*(\bar v+\lambda\nabla\kernel(\bar x))\) with \(\bar v\in\psubdiff f(\bar x)\)).
		Then, for such \(\bar y'\) one has
		\[
			f(\bar x)
		=
			\env{f}(\bar y')
		-
			\tfrac{1}{\lambda}\D(\bar x,\bar y')
		\leq
			\hull_\lambda{f}(\bar x)
		\leq
			f(\bar x)
		\]
		implying that \(f(\bar x)=\hull_\lambda{f}(\bar x)\).
		By combining this identity with \eqref{eq:envfxy} we obtain that
		\(
			\env{f}(\bar y)
		=
			f(\bar x)
			+
			\tfrac{1}{\lambda}\D(\bar x,\bar y)
		\),
		that is, that \(\bar x\in\prox_{\lambda f}(\bar y)\).

	\item
		``\ref{thm:psubdiffchar:prox} \(\Rightarrow\) \ref{thm:psubdiffchar:prox*}''
		We start by observing that nonemptiness of \(\psubdiff f(\bar x)\) follows by virtue of the equivalence ``\ref{thm:psubdiffchar:nonempty} \(\Leftrightarrow\) \ref{thm:psubdiffchar:prox}'' proven above.
		We have
		\begin{align*}
			\hull_\lambda{f}(\bar x)
		\leq
			f(\bar x)
		={} &
			\env{f}(\bar y)
			-
			\tfrac{1}{\lambda}\D(\bar x,\bar y)
		\\
		\leq{} &
			\sup_{y\in\Y}\set{
				\env{f}(y)
				-
				\tfrac{1}{\lambda}\D(\bar x,y)
			}
		\defeq
			\hull_\lambda{f}(\bar x).
		\end{align*}
		All inequalities thus hold as equalities, and in particular
		\[
			\bar y
		\in
			\argmax_{y\in\Y}\set{
				\env{f}(y)
				-
				\tfrac{1}{\lambda}\D(\bar x,y)
			}
		=
			\prox*_{-\lambda\env{f}}(\bar x)
		\]
		as claimed.
	\qedhere
	\end{itemize}
\end{proof}
The equivalence in \cref{thm:psubdiffchar} can be interpreted as a generalized ``duality'' of \(\Phi\)-conjugacy in optimal transport; \cite[\S4]{themelis2026natural}.
Equipped with \cref{thm:psubdiffchar}, one may calculate \(\psubdiff f\) easily through a geometric approach utilizing \(\conv(\lambda f+\kernel)\).

\vspace{\topsep}%
\noindent
\begin{minipage}[m]{0.55\linewidth}%
	\begin{example}\label{ex:psubdiff}%
		Let \(f(x)=x\sqrt{1-x^2}\) and \(\kernel(x)=-\sqrt{1-x^2}\).
		For \(\lambda=1\) one has that \(\conv(f+\kernel)(x)\) equals
		\[
			\begin{cases}
				(f+\kernel)(x) & \text{if } x<0
			\\
				(f+\kernel)'(0)x+(f+\kernel)(0) & \text{otherwise},
			\end{cases}
		\]
		which is smaller than \((f+\kernel)(x)\) on \((0,1]\).
		As such, \cref{thm:psubdiffchar:chull} implies that
		\[
			\psubdiff[] f(x)
		=
			\begin{cases}
				\set{\nabla f(x)} &-1<x\leq 0\\
				\emptyset &0\leq x<1.
			\end{cases}
		\]
	\end{example}
\end{minipage}
\hfill
\begin{minipage}[m]{0.42\linewidth}%
	\vspace*{0pt}%
	\includegraphics[width=\linewidth]{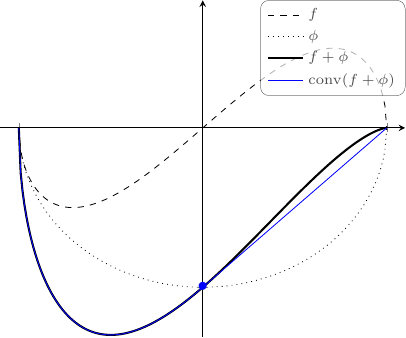}%
\end{minipage}
\vspace{\topsep}%

In convex analysis, differentiability of a convex function is closely tied with its subdifferential operator being single-valued; see for instance \cite[Prop. 17.36]{bauschke2017convex}.
Leveraging on \(\psubdiff f\), we now explore a parallel result in the absence of convexity.

\begin{theorem}\label{thm:psubdiff_single}%
	Suppose that \(\kernel\) is 1-coercive and Legendre, and let \(\func{f}{\X}{\Rinf}\) be proper, lsc, and \(\kernel\)-prox-bounded.
	Then, for any \(\lambda\in(0,\pb)\) and \(\bar x\in\interior\X\) the following are equivalent:
	\begin{enumerateq}
		\item \label{thm:psubdiff_single:single}%
		\(\psubdiff f(\bar x)\) is a singleton.
		\item \label{thm:psubdiff_single:conv}%
		\(\conv (f+\lambda^{-1}\kernel)\) is differentiable and coincides with \(f+\lambda^{-1}\kernel\) at \(\bar x\).
		\item \label{thm:psubdiff_single:hull}%
		\(\hull_\lambda{f}\) is differentiable and coincides with \(f\) at \(\bar x\).
	\end{enumerateq}
	Under any of the above equivalent conditions, one also has that
	\[
		\nabla\conv (f+\lambda^{-1}\kernel)(\bar x)
	=
		\bar u+\lambda^{-1}\nabla\kernel(\bar x)
	\quad\text{and}\quad
		\nabla\hull_\lambda{f}(\bar x)
	=
		\bar u,
	\]
	where \(\set{\bar u}=\psubdiff f(\bar x)\).
\end{theorem}
\begin{proof}~
	\begin{itemize}
	\item
		``\ref{thm:psubdiff_single:single} \(\Rightarrow\) \ref{thm:psubdiff_single:conv}''
		Let \(\psubdiff f(\bar x)=\set{\bar u}\).
		Then appealing to \cref{thm:psubdiffchar:hull} yields
		\(
			\conv(f+\lambda^{-1}\kernel)(\bar x)
		=
			(f+\lambda^{-1}\kernel)(\bar x)
		\),
		and in particular
		\[
			\lsubdiff\conv(f+\lambda^{-1}\kernel)(\bar x)
		=
			\fsubdiff(f+\lambda^{-1}\kernel)(\bar x)
		=
			\set{
				\bar u
				+
				\lambda^{-1}\nabla\kernel(\bar x)
			},
		\]
		where the first identity owes to \cref{fact:fsubdiff} and the second one to \cref{thm:FPsubdiff:f}.
		\(\lsubdiff\conv(f+\lambda^{-1}\kernel)(\bar x)\) is thus a singleton, and hence \(\conv(f+\lambda^{-1}\kernel)\) is differentiable at \(\bar x\) \cite[Thm. 25.1]{rockafellar1970convex}.

	\item ``\ref{thm:psubdiff_single:conv} \(\Leftrightarrow\) \ref{thm:psubdiff_single:hull}''
		Immediate from \cref{thm:hull}.

	\item ``\ref{thm:psubdiff_single:hull} \(\Rightarrow\) \ref{thm:psubdiff_single:single}''
		It follows from \cref{thm:hull} that
		\(
			\conv(f+\lambda^{-1}\kernel)(\bar x)
		=
			(f+\lambda^{-1}\kernel)(\bar x)
		\),
		hence that
		\(
			\fsubdiff(f+\lambda^{-1}\kernel)(\bar x)
		=
			\lsubdiff\conv(f+\lambda^{-1}\kernel)(\bar x)
		\)
		owing to \cref{fact:fsubdiff}.
		Differentiability of \(\hull_\lambda{f}\) at \(\bar x\) implies single-valuedness of these subdifferentials \cite[Thm. 25.1]{rockafellar1970convex}, hence that of \(\psubdiff f(\bar x)\) by virtue of \cref{thm:FPsubdiff:f}.
		\qedhere
	\end{itemize}
\end{proof}

We end this section with a coincidence result.

\begin{theorem}\label{thm:coincidence}%
	Suppose that \(\kernel\) is 1-coercive, and let \(\func{f_i}{\X}{\Rinf}\), \(i=1,2\) be proper, lsc, and \(\kernel\)-prox-bounded.
	Let \(\lambda\in(0,\min\set*{\pb_{f_1},\pb_{f_2}})\) be fixed, and consider the following assertions:
	\begin{enumerateq}
	\item \label{thm:coincidence:env}%
		\(\env{f}_1=\env{f}_2+c\) for some \(c\in\R\).
	\item \label{thm:coincidence:phull}%
		\(\hull_\lambda{f}_1=\hull_\lambda{f}_2+c'\) for some \(c'\in\R\).
	\item \label{thm:coincidence:prox}%
		\(\prox_{\lambda f_1}=\prox_{\lambda f_2}\).
	\item \label{thm:coincidence:sub}%
		\(\psubdiff f_1=\psubdiff f_2\).
	\end{enumerateq}
	One always has
	\ref{thm:coincidence:env}
	\(\Leftrightarrow\)
	\ref{thm:coincidence:phull}
	(in which case \(c=c'\))
	and
	\ref{thm:coincidence:prox}
	\(\Rightarrow\)
	\ref{thm:coincidence:sub}.
	If both \(f_1\) and \(f_2\) satisfy \cref{ass:range}, then
	\ref{thm:coincidence:prox}
	\(\Leftarrow\)
	\ref{thm:coincidence:sub}
	also holds.
	If instead \(\kernel\) is Legendre, then one has
	\ref{thm:coincidence:prox}
	\(\Rightarrow\)
	\ref{thm:coincidence:env}.
\end{theorem}
\begin{proof}
	We start from implications without additional assumptions.
	\begin{itemize}
	\item``\ref{thm:coincidence:env} \(\Leftrightarrow\) \ref{thm:coincidence:phull}''
		Follows from the fact that, for any \(c\in\R\) and \(\func{f}{\X}{\Rinf}\), one has
		\(
			c+\hull_\lambda{f}
		=
			\hull_\lambda{(f+c)}
		=
			-\env*{(-\env{(f+c)})}
		\)
		and the identity
		\(\env{(\hull_\lambda{f})}=\env{f}\).

	\item ``\ref{thm:coincidence:prox} \(\Rightarrow\) \ref{thm:coincidence:sub}''
		Suppose first that \(\interior\X\cap\prox_{\lambda f_1}=\interior\X\cap\prox_{\lambda f_2}\neq\emptyset\).
		Then applying \cref{thm:prox-J} completes the proof.
		Suppose now that \(\interior\X\cap\prox_{\lambda f_1}=\interior\X\cap\prox_{\lambda f_2}=\emptyset\).
		We claim that \(\psubdiff f_1=\psubdiff f_2=\emptyset\).
		If not, say there were \(u\in\psubdiff f_1(x)\) for some \(x\in\interior\X\), then the 1-coercivity of \(\kernel\) implies that
		\(\range\nabla\kernel=\R^n\) and thus \(\nabla\kernel(v)=\lambda u+\nabla\kernel(x)\) for some \(v\in\interior\X\).
		Then, by \cref{thm:prox-J} we have that \(x\in\interior\X\cap\prox_{\lambda f_1}\), which contradicts our assumption.

	\item ``\ref{thm:coincidence:sub} \(\Rightarrow\) \ref{thm:coincidence:prox}''
		If \cref{ass:range} is satisfied, the proof is immediate by using \cref{thm:prox-J} again, since the intersection with \(\interior\X\) therein becomes superfluous.

	\item ``\ref{thm:coincidence:prox} \(\Rightarrow\) \ref{thm:coincidence:env}''
		Suppose that \(\kernel\) is Legendre.
		For \(i=1,2\), let \(\func{h_i}{\R^n}{\Rinf}\) be defined as \(f_i+\lambda ^{-1}(\kernel-\j)\) on \(\X\) and \(\infty\) on \(\R^n\setminus\X\), and observe that it is lsc by virtue of \cref{thm:extlsc}.
		It also follows from \cref{thm:Euclidean} and the assumptions on \(\kernel\) that \(\prox_{\lambda f_i}\circ\nabla\kernel*=\Eprox_{\lambda h_i}\), to which \cite[Thm. 6.1]{luo2024level} yields that \(\Ehull_\lambda h_1=\Ehull_\lambda h_2+c\) for some \(c\in\R\).
		By invoking again \cref{thm:Euclidean}, note that
		\(
			\Ehull_\lambda h_i\restr_{\X}
		=
			\hull_\lambda{f}_i+\lambda^{-1}(\kernel-\j)\restr_{\X}
		\),
		\(i=1,2\).
		Since \(\kernel-\j\) is finite on \(\X=\dom\kernel\), the claimed identity \(\hull_\lambda{f}_1=\hull_\lambda{f}_2+c\) follows.
	\qedhere
	\end{itemize}
\end{proof}

	\section{Global characterizations}\label{sec:global}
		\subsection{Relative weak convexity and monotonicity}%
In this section, we explore the interplay between \(\psubdiff f\) and relative weak convexity properties of \(f\) (see \cref{def:Bcvx} below), thus establishing the Bregman counterpart of the following well-known equivalence.

\begin{fact}[{\cite[Prop. 12.19]{rockafellar1998variational} and \cite[Prop. 2.6]{chen2020proximal}}]\label{fact:Ehypo}%
	Let \(\func{f}{\R^n}{\Rinf}\) be proper, lsc, and \(\j\)-prox-bounded with threshold \(\pb^{}\).
	Then, the following are equivalent for \(\lambda\in(0,\pb^{})\):
	\begin{enumerateq}
	\item
		\(\lambda f\) is 1-weakly convex; that is, \(\lambda f+\j\) is convex.
	\item
		\(\Eprox_{\lambda f}\) is maximally monotone.
	\item
		\(\Eprox_{\lambda f}\) is convex-valued.
	\item
		\(\Eprox_{\lambda f}=(\id+\lambda\lsubdiff f)^{-1}\).
	\end{enumerateq}
\end{fact}

To this end, we will adopt the following Bregman generalization of weak convexity.
We also take the chance to mention the related extensions of strong convexity and smoothness that will be employed in the following section.

\begin{definition}\label{def:Bcvx}%
	A proper function \(\func{f}{\X}{\Rinf}\) is said to be
	\begin{enumerate}
	\item
		\B-weakly convex if \(f+\kernel\) is convex on \(\interior\X\).
	\item
		\B-strongly convex if \(f-\kernel\) is convex on \(\interior\X\).
	\item
		\B-smooth if both \(f\) and \(-f\) are \B-weakly convex, that is, if both \(f+\kernel\) and \(-f+\kernel\) are convex on \(\interior\X\).
	\end{enumerate}
	When \(\kernel=\j\), we simply say \emph{1-weakly convex}, \emph{1-strongly convex}, and \emph{1-smooth}, respectively.
\end{definition}

The choice of the prefix ``1-'' when \(\kernel=\j\) in the above definition owes to established conventions in the Euclidean setting in which this number represents the ``modulus'' of the corresponding property.
For instance, it is well known that 1-smoothness of \(\lambda f\) is equivalent to Lipschitz continuity of \(\lambda\nabla f\) with Lipschitz modulus 1.
More generally, the concept of \B-smoothness is of central importance in tackling optimization problems violating the classic Lipschitz smoothness assumption; see for instance \cite{bauschke2017descent,bolte2018first}.
The notation and terminology adopted here is borrowed from \cite{laude2023dualities}, which investigates the duality gap between relative strong convexity and smoothness for \(\kernel\) other than squared Euclidean norms.
The ``B'' in \B, short for ``Bregman'', distinguishes it from the corresponding \emph{anisotropic} counterparts that are shown to complete the duality picture.

\begin{remark}\label{rem:Bcvx:lsc}%
	Compatibly with the existing literature, starting from the seminal works \cite{bauschke2017descent,bolte2018first}, \cref{def:Bcvx} is concerned with the values of the function on \(\interior\X\), where it must be continuous, but is otherwise insensitive to its behavior on the boundary.
	In optimization contexts, it is customary to impose lower semicontinuity as a minimal working requirement.
	We deliberately avoid this a priori restriction in order to naturally view \B-smoothness as a two-sided \B-weak convexity without imposing \emph{continuity} (i.e., both lower and upper semicontinuity).
	This entails no loss of generality.
	Indeed, whenever \(f\) satisfies any of the conditions in \cref{def:Bcvx}, the same holds for the proper and lsc function \(\bar f\coloneqq\closure \mathring f\) as in \cref{rem:CQ:mathringf}---namely \(\func{\mathring f}{\X}{\Rinf}\) with \(\mathring f(x)=f(x)\) for \(x\in\interior\X\) and \(\infty\) otherwise---which agrees with \(f\) on \(\interior\X\).
\end{remark}

We start with a simple result that will be useful in the proof of the following \cref{thm:hypo}.

\begin{lemma}\label{thm:cxvridom}%
	Let \(\func{h}{\R^n}{\Rinf}\) be a proper and lsc function.
	Then, \(h\) is convex iff \(h\equiv\conv*h\) on \(\relint\dom\conv* h\).
\end{lemma}
\begin{proof}
	The ``only if'' implication is obvious.
	Suppose now that \(h\equiv\conv*h\) on \(\relint\dom\conv*h\).
	Since \(h\) is proper, \(\conv*h\not\equiv\infty\) and as such \(\dom\conv* h\) is nonempty (and convex), and thus so is \(\relint\dom\conv*h\) \cite[Thm. 6.2]{rockafellar1970convex}.
	Fix \(x_0\in\relint\dom\conv*h\) and
	\(
		x_1\in\conv\dom h\subseteq\dom\conv*h
	\),
	and for \(\mu\in\R\) define \(x_\mu\coloneqq (1-\mu)x_0+\mu x_1\).
	Then,
	\[
		h(x_1)
	\leq
		\liminf_{\mu\to1^-}
		h(x_\mu)
	=
		\lim_{\mu\to1^-}
		\conv*h(x_\mu)
	=
		\conv*h(x_1)
	<
		\infty,
	\]
	where the first inequality owes to lower semicontinuity, the last equality follows from \cite[Thm. 7.5]{rockafellar1970convex}
	(even though \(\conv*h\) could, a priori, be improper, the reference still holds since \(x_1\in\dom\conv*h\)), and the last inequality from the fact that \(x_1\in\dom\conv*h\).
	This shows that \(h(x_1)=\conv*h(x_1)<\infty\) for \(x_1\in\conv\dom h\).
	From the arbitrariness of \(x_1\in\conv\dom h\), it follows that \(\conv\dom h=\dom h\) and that \(h\) agrees with a convex function on its convex domain, and must thus be convex.
\end{proof}

The following is the anticipated extension of \cref{fact:Ehypo} to the Bregman setting.
While the equivalence ``\ref{thm:hypo:hypo} \(\Leftrightarrow\) \ref{thm:hypo:maxmono}'' has been discovered by Kan and Song in \cite[Thm. 4.2]{kan2012moreau} and the implication ``\ref{thm:hypo:hypo} \(\Rightarrow\) \ref{thm:hypo:proxcvx}'' by Wang and Bauschke \cite[Prop. 2.10]{wang2022bregman}, the other implications listed below, to the best of our knowledge, are new.
In addition, we re-propose a self-contained proof of the known implication to account for the possible lack of lower semicontinuity, in our setting, of function \(f\) when extended to \(\R^n\).
Readers are referred to \cref{fig:hypo} for its connection to the well-known \cref{fact:Ehypo}.
Moreover, the following reveals that \(\dom\psubdiff f\) is ``big'' iff the underlying function \(\lambda f\) is \B-weakly convex.

\begin{theorem}\label{thm:hypo}%
	Suppose that \(\kernel\) is Legendre and 1-coercive, and let \(\func{f}{\X}{\Rinf}\) be proper, lsc, and \(\kernel\)-prox-bounded.
	Fix \(\lambda\in(0,\pb)\) and consider the following statements:
	\begin{enumerateq}
	\item \label{thm:hypo:hypo}%
		\(f+\lambda^{-1}\kernel\) is convex.

	\item \label{thm:hypo:hull}%
		\(f=\hull_\lambda{f}\).

	\item \label{thm:hypo:maxmono}%
		\(\prox_{\lambda f}\circ\nabla\kernel*\) is maximally monotone.

	\item \label{thm:hypo:proxcvx}%
		\(\prox_{\lambda f}\circ\nabla\kernel*\) is convex-valued.

	\item \label{thm:hypo:subdiff}%
		\(\relint\conv\dom f\subseteq\dom\lsubdiff f\) and \(\lsubdiff f=\psubdiff f\) on \(\interior\X\).

	\item \label{thm:hypo:dom}%
	\(\relint\conv\dom f\cap\interior\X\subseteq\dom\psubdiff f\).

	\item \label{thm:hypo:J}%
		\(\relint\conv\dom f\subseteq\dom\lsubdiff f\) and \(\prox_{\lambda f}=(\nabla\kernel+\lambda\lsubdiff f)^{-1}\circ\nabla\kernel\).
	\end{enumerateq}
	Then,
	\ref{thm:hypo:J}
	\(\Rightarrow\)
	\ref{thm:hypo:hypo}
	\(\Leftrightarrow\)
	\ref{thm:hypo:hull}
	\(\Leftrightarrow\)
	\ref{thm:hypo:maxmono}
	\(\Leftrightarrow\)
	\ref{thm:hypo:proxcvx}
	\(\Rightarrow\)
	\ref{thm:hypo:subdiff}
	\(\Rightarrow\)
	\ref{thm:hypo:dom}
	always hold.

	The implication
	\ref{thm:hypo:dom}
	\(\Rightarrow\)
	\ref{thm:hypo:hypo}
	holds when \(\relint\conv \dom f\subseteq\interior\X\), and so does
	\ref{thm:hypo:J}
	\(\Leftarrow\)
	\ref{thm:hypo:hypo}
	if in addition \cref{ass:range} is satisfied.
\end{theorem}
\begin{proof}
	We first prove the implications without additional assumptions.
	It is helpful to remind that \(\dom f=\dom(\lambda f+\kernel)\), and that under the assumptions in the statement \(\lambda\tilde f+\kernel\) is 1-coercive, where \(\func{\tilde f}{\R^n}{\Rinf}\) is the canonical extension of \(f\) as in \cref{def:fext}, namely such that \(\tilde{f}(x)=f(x)\) for \(x\in\X\) and \(\tilde{f}(x)=\infty\) for \(x\in\R^n\setminus\X\).
	\begin{itemize}
	\item ``\ref{thm:hypo:J} \(\Rightarrow\) \ref{thm:hypo:hypo}''
		One has that
		\[
			(\nabla\kernel+\lambda\lsubdiff f)^{-1}
		=
			\prox_{\lambda f}\circ\nabla\kernel*
		=
			\fsubdiff(\kernel+\lambda f)^{-1},
		\]
		where the last identity follows from \cref{thm:Jfsubdiff}, overall implying that
		\(
			\fsubdiff(\kernel+\lambda f)
		=
			\nabla\kernel+\lambda\lsubdiff f
		\).
		Hence
		\(
			\fsubdiff(\lambda f+\kernel)\neq\emptyset
		\)
		on
		\(
			\relint\conv\dom(\lambda f+\kernel)
		=
			\relint\conv\dom f\subseteq\dom\lsubdiff f
		\),
		which entails through \cref{fact:fsubdiff} that \(f+\tfrac{1}{\lambda}\kernel=\conv*(f+\tfrac{1}{\lambda}\kernel)\) on \(\relint\conv\dom(f+\tfrac{1}{\lambda}\kernel)\).
		Invoking \cref{thm:cxvridom} concludes the proof.

	\item ``\ref{thm:hypo:hypo} \(\Leftrightarrow\) \ref{thm:hypo:hull}''
		Follows from \cref{thm:hull}.

	\item
		``\ref{thm:hypo:hypo} \(\Rightarrow\) \ref{thm:hypo:maxmono}''
		The assumption together with \cref{thm:extlsc} ensures that \(\lambda\tilde{f}+\kernel\) is proper lsc (on \(\R^n\)) and convex.
		Hence
		\[
			\prox_{\lambda f}\circ\nabla\kernel*
		=
			\prox_{\lambda \tilde{f}}\circ\nabla\kernel*
		=
			\bigl(
				\lsubdiff(\lambda\tilde{f}+\kernel)
			\bigr)^{-1},
		\]
		where the last equality follows from \cref{thm:Jfsubdiff} and the convexity of \(\lambda\tilde{f}+\kernel\).
		The operator \(\lsubdiff(\lambda\tilde{f}+\kernel)\) is maximally monotone in view of \cite[Thm. 12.17]{rockafellar1998variational}, and therefore so is its inverse.

	\item
		``\ref{thm:hypo:maxmono} \(\Leftrightarrow\) \ref{thm:hypo:proxcvx}''
		The right implication is clear.
		The converse direction follows from \cref{thm:proxhull} together with the fact that \(\hull f+\lambda^{-1}\kernel\) is convex and the consequent maximal monotonicity of \(\prox_{\lambda \hull}\circ\nabla\kernel*\); see the proven implication ``\ref{thm:hypo:hypo} \(\Rightarrow\) \ref{thm:hypo:maxmono}'' or \cite[Thm. 4.2]{kan2012moreau}.

	\item ``\ref{thm:hypo:proxcvx} \(\Rightarrow\) \ref{thm:hypo:hypo}''
		It suffices to show that \(\lambda\tilde{f}+\kernel\) is convex.
		Note that
		\begin{align*}
			\fsubdiff(\lambda\tilde{f}+\kernel)
		&=
			\left(\prox_{\lambda\tilde{f}}\circ\nabla\kernel*\right)^{-1}
		=
			\left(\prox_{\lambda\hull_\lambda{\tilde{f}}}\circ\nabla\kernel*\right)^{-1}\\
		&=
			\fsubdiff(\lambda\hull_\lambda{\tilde{f}}+\kernel)
		=
			\lsubdiff\conv(\lambda\tilde{f}+\kernel),
		\end{align*}
		where the first and the third equalities follow from \cref{thm:Jfsubdiff}, while the second one owes to \ref{thm:hypo:proxcvx} and \cref{thm:proxhull}.
		We claim that \(\fsubdiff(\lambda\tilde{f}+\kernel)\neq\emptyset\) on \(\relint\conv\dom f\).
		In turn, \(\lambda\tilde{f}+\kernel=\conv(\lambda\tilde{f}+\kernel)\) on \(\relint\conv\dom f\) and therefore \(\lambda\tilde{f}+\kernel\) is convex owing to \cref{thm:cxvridom}.
		Indeed, notice that
		\(
			\dom\conv(\lambda\tilde{f}+\kernel)
		=
			\conv\dom(\lambda\tilde{f}+\kernel)
		=
			\conv(\dom\tilde{f}\cap\dom\kernel)
		=
			\conv\dom f
		\).
		Therefore \(\relint\conv\dom f=\relint\dom\conv(\lambda\tilde{f}+\kernel)\subseteq \dom\lsubdiff \conv(\lambda\tilde{f}+\kernel)\).

	\item ``\ref{thm:hypo:hypo} \(\Rightarrow\) \ref{thm:hypo:subdiff}''
		The convexity of \(f+\lambda^{-1}\kernel\) enforces \(\dom f\) to be convex.
		We consider two cases:
		\begin{itemize}
		\item
			Assume first that \(\dom f\not\subseteq\X\setminus\interior\X\), so that
			\(\relint\conv\dom f\subseteq\relint\X=\interior\X\) by \cite[Cor. 6.5.2]{rockafellar1970convex}.
			So for \(x\in\interior\X\),
			\(
				\lsubdiff f(x)
			=
				\lsubdiff(f+\lambda^{-1}\kernel)(x)-\lambda^{-1}\nabla\kernel(x)
			=
				\fsubdiff(f+\lambda^{-1}\kernel)(x)-\lambda^{-1}\nabla\kernel(x)
			=
				\psubdiff f(x)
			\)
			by \cref{thm:FPsubdiff:f}.
			Moreover, the convexity of \(f+\lambda^{-1}\kernel\) ensures that in particular for \(x\in\relint\dom f\subseteq\interior\X\), \(\lsubdiff f(x)=\lsubdiff(f+\lambda^{-1}\kernel)(x)-\lambda^{-1}\nabla\kernel(x)\neq\emptyset\).

		\item
			Suppose instead that \(\dom f\subseteq\X\setminus\interior\X\) holds.
			In this case, the condition \(\lsubdiff f=\psubdiff f\) on \(\interior\X\) holds vacuously, since \(\dom f\cap\interior\X=\emptyset\) in this case.
			To show nonemptiness of \(\lsubdiff f\) on \(\relint\conv\dom f\), without loss of generality (that is, up to replacing \(f\) with \(f({}\cdot{}+x_0)\) for some \(x_0\in\dom f\)), let the affine hull of \(\dom f\) be a vector subspace \(V\) of \(\R^n\) of dimension \(m\leq n\).
			Let \(\func{L}{V}{\R^m}\) be an invertible linear transformation, and define \(\func{\hat f}{\R^m}{\Rinf}\) and \(\func{\hat\kernel}{\R^m}{\Rinf}\) as \(\hat f=\tilde f\circ L^{-1}\) and \(\hat\kernel=\kernel\circ L^{-1}\).
			Then, both \(\lambda\hat f+\hat\kernel\) and \(\hat\kernel\) are convex functions on \(\R^m\) with
			\(
				\interior\dom\hat\kernel
			\supseteq
				\interior\dom\hat f
			=
				\relint\dom\hat f
			=
				\relint(L\dom f)
			=
				L\relint\dom f
			\),
			where the last identity follows from \cite[Thm. 6.6]{rockafellar1970convex}.
			Since \(\hat\kernel\) is convex, it is strictly continuous on \(\interior\dom\hat\kernel\supseteq L\relint\dom f\), hence
			\(
				\emptyset
			\neq
				\fsubdiff(\lambda\hat f+\hat\kernel)
			=
				\lsubdiff(\lambda\hat f+\hat\kernel)
			\subseteq
				\lambda\lsubdiff\hat f+\lsubdiff\hat\kernel
			\)
			on \(L\relint\dom f\), with inclusion owing to \cite[Ex. 10.10]{rockafellar1998variational}.
			Since \(\fsubdiff(\lambda\hat f+\hat\kernel)\) is nonempty there, so is \(\lsubdiff\hat f\).
			This shows that \(L\relint\dom f\subseteq\dom\lsubdiff\hat f=L\dom\lsubdiff f\).
			Since \(L\) is invertible, the claimed inclusion \(\relint\conv\dom f=\relint\dom f\subseteq\dom\lsubdiff f\) follows.
		\end{itemize}

	\item ``\ref{thm:hypo:subdiff} \(\Rightarrow\) \ref{thm:hypo:dom}''
		Suppose that \(\relint\conv\dom f\cap\interior\X\neq\emptyset\), otherwise the claim holds vacuously.
		Then for \(x\in \relint\conv\dom f\cap\interior\X\), \ref{thm:hypo:subdiff} implies \(\psubdiff f(x)=\lsubdiff f(x)\neq\emptyset\).
	\end{itemize}

	\noindent
	In what follows, suppose that \(\relint\conv\dom f\subseteq\interior\X\).

	\begin{itemize}
	\item
		``\ref{thm:hypo:dom} \(\Rightarrow\) \ref{thm:hypo:subdiff}''
		The assumption in \ref{thm:hypo:dom} then implies that
		\[
			\psubdiff f
		=
			\fsubdiff(f+\lambda^{-1}\kernel)-\lambda^{-1}\nabla\kernel
		\neq
			\emptyset
		\quad
			\text{on }
			\relint\conv\dom f,
		\]
		and therefore \(\conv*(f+\lambda^{-1}\kernel)=\conv(f+\lambda^{-1}\kernel)=(f+\lambda^{-1}\kernel)\) on \(\relint\conv\dom f\) by \cref{thm:cvxext,fact:fsubdiff}.
		Then \(f+\lambda^{-1}\kernel\) is convex in view of \cref{thm:cxvridom}.
		So for \(x\in\interior\X\),
		\begin{align*}
			\lsubdiff f(x)
			&=
				\lsubdiff(f+\lambda^{-1}\kernel)(x)-\lambda^{-1}\nabla \kernel(x)
			=
				\fsubdiff(f+\lambda^{-1}\kernel)(x)-\lambda^{-1}\nabla \kernel(x)\\
			&=
				\psubdiff f(x),
		\end{align*}
		which means that \(\relint\dom f=\relint\dom(f+\lambda^{-1}\kernel)\subseteq\dom\lsubdiff f\) and \(\lsubdiff f=\psubdiff f\) on \(\interior\X\) as desired.

	\item
		``\ref{thm:hypo:subdiff} \(\Rightarrow\) \ref{thm:hypo:hypo}''
		Observe that \(\relint\conv\dom f=\relint\dom\conv f\neq\emptyset\); the identity follows from \cite[Prop. 9.8(iv)]{bauschke2017convex} and \cite[Thm. 6.3]{rockafellar1970convex}, and nonemptiness from properness of \(f\) together with \cite[Thm. 6.2]{rockafellar1970convex}.
		Pick \(x_0\in\relint\conv\dom f=\relint\dom\conv(\lambda f+\kernel)\subseteq\interior\X\).
		Then,
		\[
			\emptyset
		\neq
			\lsubdiff f(x_0)
		=
			\psubdiff f(x_0)
		=
			\fsubdiff[f+\tfrac{1}{\lambda}\kernel](x_0)
			-
			\tfrac{1}{\lambda}\nabla\kernel(x_0),
		\]
		where the second identity again uses \cref{thm:FPsubdiff:f}.
		This shows that \(\fsubdiff[f+\tfrac{1}{\lambda}\kernel]\) agrees with \(\lsubdiff[f+\tfrac{1}{\lambda}\kernel]\) on \(\relint\dom f\) and is nonempty there, which appealing to \cref{fact:fsubdiff} implies that \(f+\tfrac{1}{\lambda}\kernel=\conv*(f+\tfrac{1}{\lambda}\kernel)\) on \(\relint\dom\conv(f+\tfrac{1}{\lambda}\kernel)\).
		Invoking \cref{thm:cxvridom} concludes the proof.
	\end{itemize}

	\noindent
	To conclude, let us also assume that \cref{ass:range} holds.
	\begin{itemize}
	\item ``\ref{thm:hypo:hypo} \(\Rightarrow\) \ref{thm:hypo:J}''
		Appealing to \cref{thm:Jfsubdiff}, one has that \(\bar x\in\prox_{\lambda f}(\bar y)\) iff \(\nabla\kernel(\bar y)\in\fsubdiff(\lambda f+\kernel)(\bar x)\), and such subdifferential equals \(\lsubdiff(\lambda f+\kernel)(\bar x)\) by convexity.
		Since \(\bar x\in\interior\X\) by \cref{ass:range}, it also coincides with \(\nabla\kernel(\bar x)+\lambda\lsubdiff f(\bar x)\), leading to the sought conclusion.
	\qedhere
	\end{itemize}
\end{proof}

\begin{remark}
	We make some comments on the conditions in \cref{thm:hypo}:
	\begin{enumerate}
	\item
		The additional assumption \(\relint\conv \dom f\subseteq\interior\X\) is unavoidable for the implication ``\ref{thm:hypo:subdiff} \(\Rightarrow\) \ref{thm:hypo:hypo}''.
		Consider \(\func{\kernel}{\R^2}{\Rinf}\) given by
		\begin{align*}
			\kernel(x)
		={} &
			\begin{cases}
				x_1\ln x_1+x_2\ln x_2 & \text{if }x\geq0
			\\
				\infty & \text{otherwise,}
			\end{cases}
		\shortintertext{and \(\func{f}{\X=\R_+^2}{\Rinf}\) defined as}
			f(x)
		={} &
			\begin{cases}
				\sqrt{1-(1-x_1)^2} & \text{if } x_2=0
			\\
				\infty & \text{otherwise,}
			\end{cases}
		\end{align*}
		with the convention that \(0\ln(0)=0\).
		One has \(\relint\conv\dom f=\R_{++}\times\set{0}\) while \(\interior\X=\R_{++}\times\R_{++}\), so that \(\relint\conv\dom f\not\subseteq\interior\X\).
		For every \(\lambda>0\), the conditions in \cref{thm:hypo:subdiff} are met, having \(\lsubdiff f=\psubdiff f=\emptyset\) on \(\interior\X\) and
		\(
			\relint\conv\dom f
		=
			\R_{++}\times\set{0}
		\subseteq
			\R_{+}\times\set{0}
		=
			\dom\lsubdiff f
		\).
		However, \cref{thm:hypo:hypo} is not verified, as \(f+\lambda^{-1}\kernel\) is clearly nonconvex.

	\item
		When \(\X=\R^n\), the assumption \(\relint\conv \dom f\subseteq\dom\lsubdiff f\) in the statements of \cref{thm:hypo:subdiff,thm:hypo:J} is superfluous.

		Regarding \cref{thm:hypo:subdiff}, having \(\lsubdiff f=\psubdiff f\) everywhere implies through \cref{thm:FPsubdiff:f} that
		\(
			\lsubdiff(f+\lambda^{-1}\kernel)
		=
			\fsubdiff(f+\lambda^{-1}\kernel)
		\)
		everywhere, which further entails the monotonicity of \(\lsubdiff(f+\lambda^{-1}\kernel)\) and equivalently the convexity of \(f+\lambda^{-1}\kernel\).
		Consequently, \(\dom f\) is convex and
		\(
			\lsubdiff f
		=
			\lsubdiff(f+\lambda^{-1}\kernel)
		-
			\lambda^{-1}\nabla\kernel
		\neq
			\emptyset
		\)
		on \(\relint\dom f=\relint\conv\dom f\).

		Similarly, when \(\X=\R^n\) the condition \(\prox_{\lambda f}=(\nabla\kernel+\lambda\lsubdiff f)^{-1}\circ\nabla\kernel\) in \cref{thm:hypo:J} implies through \cref{thm:monotone} that the operator \(\lsubdiff(\lambda f+\kernel)\) is monotone on the whole \(\R^n\), whence the convexity of \(\lambda f+\kernel\).
		In turn, the inclusion \(\relint\conv \dom f\subseteq\dom\lsubdiff f\) readily follows.

	\item
		\Cref{ass:range} is needed for the implication ``\ref{thm:hypo:J} \(\Leftarrow\) \ref{thm:hypo:hypo}''.
		Consider any kernel \(\kernel\) whose domain \(\X\) is not open, and let \(f=\indicator_{\set{\bar x}}\restr_{\X}\) for some \(\bar x\in\X\setminus\interior\X\).
		Then, \(f+\lambda^{-1}\kernel\) is trivially convex and \(\prox_{\lambda f}\equiv\set{\bar x}\).
		However, \(\bar x\notin\range(\nabla\kernel+\lambda\lsubdiff f)^{-1}\circ\nabla\kernel\) because \(\bar x\in\X\setminus\interior\X\).
	\end{enumerate}
\end{remark}

\subsection[Convexity, B- and a-firm nonexpansiveness]{Convexity, \B- and \a-firm nonexpansiveness}\label{sec:FNE}%
In this section, we investigate a non-Euclidean counterpart of the following classic correspondence; see \cite[Thm. 12.17]{rockafellar1998variational}, \cite[Thm. 1]{rockafellar2021characterizing}, and \cite[Thm. 3.17]{wang2010chebyshev} for a proof.

\begin{fact}\label{fact:classic}%
	Let \(\func{f}{\R^n}{\Rinf}\) be a proper, lsc, and \(\j\)-prox-bounded function.
	Then the following are equivalent:
	\begin{enumerateq}
	\item \label{fact:classic:cvx}%
		\(f\) is convex.
	\item \label{fact:classic:mono}%
		\(\lsubdiff f\) is (maximally) monotone.
	\item \label{fact:classic:envcvx}%
		\((\exists \lambda>0)\) \(\Eenv_\lambda{f}\) is convex.
	\item \label{fact:classic:fne}%
		\((\exists \lambda>0)\) \(\Eprox_{\lambda f}\) is firmly nonexpansive.
	\item \label{fact:classic:envcoco}%
		\((\exists \lambda>0)\) \(\Eenv_\lambda f\) is differentiable and \(\nabla(\lambda\Eenv_\lambda f)\) is FNE.
	\end{enumerateq}
\end{fact}

The same correspondence does not hold for Bregman proximal operators and Moreau envelopes, owing to a duality gap pointed out by Laude et al. \cite{laude2023dualities}.
This gap makes it necessary to introduce two distinct non-Euclidean extensions of firm nonexpansiveness; following the terminology of \cite{laude2023dualities}, we refer to these as the \emph{Bregman (\B)} and \emph{anisotropic (\a)} variants; see \cref{def:BFNE,def:aFNE} respectively.

\begin{definition}[\B-firm nonexpansiveness {\cite[Def. 2.5(iii)]{borwein2011characterization}}]\label{def:BFNE}%
	Let \(\kernel\) be a dgf with \(\X=\dom\kernel\) and \(\Y=\interior\dom\kernel\).
	An operator \(\func{T}{\Y}{\X}\) is said to be \emph{\B-firmly nonexpansive (\B-FNE)} if
	\(\range T\subseteq\interior\X\) and\footnote{%
		One may wonder why we do not simply consider operators \(\func{T}{\Y}{\Y}\), which would render the range specification unnecessary.
		The chosen convention, allowing \(\func{T}{\Y}{\X}\) with \(\range T \subseteq \interior\X\), is intentionally adopted to remain fully compatible with left Bregman proximal mappings; see also \cref{rem:aFNE:X}.%
	}%
	\begin{equation}\label{eq:BFNE}
		\DD(T\bar y,Ty)
	\leq
		\innprod{
			T\bar y
			-
			Ty
		}{
			\nabla\kernel(\bar y)
			-
			\nabla\kernel(y)
		}
	\quad
		\forall\bar y,y\in\Y.
	\end{equation}
\end{definition}

This definition generalizes the usual notion of firm nonexpansiveness, recovered when \(\kernel=\j\), and is closely related to the \(\nabla\kernel\)-firm nonexpansiveness as in \cite[Def. 4.1]{wang2022bregman}.
In particular, we point out that, when \(\kernel\) is Legendre, \B-firm nonexpansiveness of \(T\) amounts to \(\nabla\kernel\)-firm nonexpansiveness of \(T\circ\nabla\kernel*\).
\B-firmly nonexpansive operators, possibly set-valued, are referred to as ``\(\D\)-firm'' in \cite[Def. 3.4]{bauschke2003bregman}.
Our restriction to single-valued mappings, same as in \cite{borwein2011characterization}, entails no loss of generality, since single-valuedness is automatically ensured whenever \(\kernel\) is strictly convex on the interior of its domain; see \cite[Prop. 3.5(iii)]{bauschke2003bregman}.
The chosen terminology is compatible with the acronym ``BFNE'' used in \cite{borwein2011characterization}, and is further inspired by \cite{laude2023dualities} to distinguish this notion from the complementary \emph{anisotropic} concept which we abbreviate as ``\a''.

\begin{definition}[\a-firm nonexpansiveness]\label{def:aFNE}%
	Let \(\kernel\) be a dgf with \(\X=\dom\kernel\) and \(\Y=\interior\dom\kernel\).
	We say that \(\func{T}{\Y}{\X}\) is \emph{\a-firmly nonexpansive} (\emph{\a-FNE}) if
	\begin{equation}\label{eq:aFNE}
	\kernel(\bar y)
	+
	\kernel(y)
\geq
	\kernel((\id-T)\bar y+Ty)
	+
	\kernel((\id-T)y+T\bar y)
\quad
	\forall \bar y,y\in\Y.
	\end{equation}
\end{definition}

By looking at the right-hand side of \eqref{eq:aFNE} one may wonder whether the definition only makes sense when \(\kernel\) has full domain.
Surprisingly, however, gradients of \B-smooth convex functions \(\func{f}{\X}{\R}\), when composed with \(\nabla\kernel*\), constitute a prominent nontrivial class of \a*-FNE operators, as we will see in \cref{thm:Bcoco}.
A simple surprising consequence of this result is the fact that any combination of the form \(\nabla\kernel(\bar y)+\nabla f(y)-\nabla f(\bar y)\) with \(\bar y,y\in\interior\dom\kernel\) always belongs to \(\dom\kernel*\).
We defer the details to \cref{thm:Bcoco}, and first draw some observations.

\begin{remark}\label{rem:aFNE}%
	Some comments are in order.
	\begin{enumerate}
	\item
		\emph{\a-/\B-FNE vs FNE.}
		By suitably expanding the squares, it is easy to see that the inequalities defining \B- and \a-FNE mappings both reduce to the classical firm nonexpansiveness (i.e., 1-cocoercivity) inequality \(\innprod{T\bar y-Ty}{\bar y-y}\geq\norm{T\bar y-Ty}^2\) when \(\kernel=\j\).
		For general \(\kernel\), however, the two notions no longer coincide.
		This mismatch is precisely what underlies the gap highlighted in \cref{fig:cvx} between the convexity of a function and that of its envelope; see \cref{sec:dualitygap}.

	\item \label{rem:aFNE:X}%
		\emph{Target set \(\X\).}
		Compliance with \eqref{eq:BFNE} forces the image of a \B-FNE operator to lie in (the interior of) \(\X\).
		For \a-FNE operators, the same holds up to a translation, in a way that parallels the relationship between vector subspaces and affine subspaces in linear algebra.
		Specifically, if \(T\) admits a fixed point \(\bar y\in\Y\) (i.e., \(T\bar y=\bar y\)), then \eqref{eq:aFNE} directly forces \(T(\Y)\subseteq\dom\kernel\), since the arguments of \(\kernel\) on the right-hand side reduce to \(Ty\) and \(\bar y\) respectively.
		In the absence of a fixed point, one can always pass to the translated operator \(\hat T \coloneqq T + \hat y - T\hat y\) for any reference point \(\hat y \in \Y\): the operator \(\hat T\) satisfies \eqref{eq:aFNE} whenever \(T\) does, and \(\hat y\) is by construction a fixed point of \(\hat T\).\
		Restricting the target set to \(\X\) is thus natural both in our Bregman proximal map setting and from a fixed-point perspective---the latter being a central focus in analyses involving firm nonexpansiveness---just as it is natural in linear algebra to work with a vector subspace rather than an arbitrary affine subspace when a distinguished basepoint is available.

	\item \label{rem:aFNE:singlevalued}%
		\emph{Set-valued extensions.}
		\Cref{def:aFNE} admits a natural extension to set-valued operators.
		As in the case of \B-firm nonexpansiveness, this extension introduces no additional generality whenever \(\kernel\) is strictly convex on \(\Y\), as the operator must be single-valued in this case (the argument follows along the lines of the proof of \cref{thm:aFNE:C0}).
	\end{enumerate}
\end{remark}

\begin{lemma}[main properties of \a-FNE operators]\label{thm:aFNE:mainprop}%
	Let \(\kernel\) be a dgf with \(\X=\dom\kernel\) and \(\Y=\interior\dom\kernel\), and let \(\func{T}{\Y}{\X}\) be \a-FNE.
	Then, the following hold:
	\begin{enumerate}
	\item \label{thm:aFNE:I-T}%
		\(\id-T\) is \a-FNE.
	\item \label{thm:aFNE:Tmonotone}%
		\(T\) is \(\nabla\kernel\)-monotone in the sense of \cite[Def. 2.6]{borwein2011characterization}, namely
		\[
			\innprod{T\bar y-Ty}{\nabla\kernel(\bar y)-\nabla\kernel(y)}
		\geq
			0
		\quad
			\forall \bar y, y\in\Y.
		\]
		In particular, when \(\kernel\) is Legendre, \(\func{S\coloneqq T\circ\nabla\kernel*}{\Y*}{\X}\) and \(\nabla\kernel*-S\) are monotone mappings, and thus
		\[
			0
		\leq
			\innprod{S\bar\eta-S\eta}{\bar\eta-\eta}
		\leq
			\DD*(\bar\eta,\eta)
		\quad
			\forall \bar\eta,\eta\in\Y*.
		\]
	\item \label{thm:aFNE:C0}%
		If \(\kernel\) is strictly convex on \(\Y\), then \(T\) is continuous.
	\end{enumerate}
\end{lemma}
\begin{proof}~
	\begin{itemize}
	\item ``\ref{thm:aFNE:I-T}''
		Clear after observing that \eqref{eq:aFNE} is unaffected by replacing \(T\) with \(\id-T\).

	\item ``\ref{thm:aFNE:Tmonotone}''
		Let \(\bar x\coloneqq T\bar y\) and \(x\coloneqq Ty\).
		Lower bounding the right-hand side of \eqref{eq:aFNE} with gradient inequalities yields that
		\[
			\kernel(\bar y)
			+
			\kernel(y)
		\geq
			\kernel(\bar y+x-\bar x)
			+
			\kernel(y+\bar x-x)
		\geq
			\kernel(\bar y)
			+
			\kernel(y)
			+
			\innprod{\nabla\kernel(\bar y)-\nabla\kernel(y)}{x-\bar x}.
		\]
		Suitably rearranging yields the claimed \(\nabla\kernel\)-monotonicity.
		In turn, the claim on the mapping \(S\) when \(\kernel\) is Legendre follows from the \(\nabla\kernel\)-monotonicity of \(T\) and \(\id-T\), owing to the previous point, and consequent monotonicity of \(S\) and \(\nabla\kernel*-S\).

	\item ``\ref{thm:aFNE:C0}''
		Adopting the notation of the above proof, note that
		\[
			\tfrac{1}{2}
			\kernel(\bar y)
			+
			\tfrac{1}{2}
			\kernel(y)
		\geq
			\tfrac{1}{2}
			\kernel(\bar y+x-\bar x)
			+
			\tfrac{1}{2}
			\kernel(y+\bar x-x)
		\geq
			\kernel\bigl(\tfrac{\bar y+y}{2}\bigr),
		\]
		where the first inequality is \a-firm nonexpansiveness, and the second one convexity of \(\kernel\).
		By observing that the second inequality is strict whenever \(\kernel\) is strictly convex and \(x-\bar x\neq 0\), a simple limiting argument with \(y\to\bar y\) proves the claimed continuity.
	\qedhere
	\end{itemize}
\end{proof}

We will now show that gradients of \B*-smooth convex functions constitute a prominent class of \a-FNE operators when composed with \(\nabla\kernel\).
Remarkably, this is shown to be true without 1-coercivity of neither \(\kernel\) nor \(\kernel*\), which confirms the significance of \cref{def:aFNE} even when \(\kernel\) does not have full domain.
Similarly to the Euclidean case, \a-firm nonexpansiveness of the gradient turns out to be a ``symmetrized'' version of a cocoercivity inequality for the function, first discovered in \cite[Prop. 5.5]{dragomir2021methodes} in the Bregman setting and recently generalized in \cite[Thm. 4.10]{themelis2026natural}.
Not only does the result below extend these developments beyond the 1-coercive case, but it also resolves an open question posed in \cite{themelis2026natural}, namely whether the cocoercivity inequality alone suffices to guarantee \B-smoothness.
As we show, \a-firm nonexpansiveness provides the missing link.

\begin{theorem}[characterization of convex \B-smooth functions]\label{thm:Bcoco}%
	Let \(\func{\kernel}{\R^n}{\Rinf}\) be Legendre and \(\func{f}{\X}{\Rinf}\) be proper.
	Then, the following are equivalent:
	\begin{enumerateq}
	\item \label{thm:Bcoco:Bsmooth}%
		\(f\) is convex and \B-smooth.
	\item \label{thm:Bcoco:Bcoco}%
		\(f\) is differentiable on \(\interior\X\) and satisfies the \emph{Bregman cocoercivity inequality}
		\begin{equation}\label{eq:Bcoco}
			\D_f(\bar x,x)
		\geq
			\D*\bigl(
				\nabla\kernel(\bar x)-(\nabla f(\bar x)-\nabla f(x))
				,\,
				\nabla\kernel(\bar x)
			\bigr)
		\quad
			\forall x,\bar x\in\interior\X.
		\end{equation}
	\item \label{thm:Bcoco:a*FNE}%
		\(f\) is differentiable on \(\interior\X\) and \(\nabla f\circ\nabla\kernel*\) is \a*-FNE.\footnote{%
			Strictly speaking, to ensure that the operator's range is in \(\X*\) in compliance with \cref{def:aFNE}, one may need to consider the translated operator \(\nabla f\circ\nabla\kernel*+\nabla[\kernel-f](\hat y)\) for some fixed \(\hat y\in\Y\); see \cref{rem:aFNE:X}.
		}%
	\end{enumerateq}
	In particular, under any of the above equivalent conditions it holds that
	\begin{equation}\label{eq:Bcoco:X*}
		\nabla\kernel(\bar x)-(\nabla f(\bar x)-\nabla f(x))
	\in
		\dom\kernel*
	\quad
		\forall \bar x,x\in\interior\X.
	\end{equation}
\end{theorem}
\begin{proof}
	We begin by observing that the inclusion in \eqref{eq:Bcoco:X*} follows from assertion \ref{thm:Bcoco:a*FNE}, as discussed in \cref{rem:aFNE:X}.
	We now prove the claimed equivalences.
	\begin{itemize}
	\item
		``\ref{thm:Bcoco:Bsmooth} \(\Rightarrow\) \ref{thm:Bcoco:Bcoco}''
		For \(\varepsilon>0\) let \(\kernel_\varepsilon\coloneqq\kernel+\varepsilon\j\), and note that \(f\) is \B_{\kernel_\varepsilon}-smooth.
		Since \(\kernel_\varepsilon\) is Legendre and 1-coercive, it follows from \cite[Prop. 5.5]{dragomir2021methodes} that
		\[
			\D_f(\bar x,x)
		\geq
			\D_{\kernel*_\varepsilon}\bigr(\overbracket[0.5pt]{\nabla\kernel_\varepsilon(\bar x)+[\nabla f(x)-\nabla f(\bar x)]}^{v_\varepsilon},\nabla\kernel_\varepsilon(\bar x)\bigr)
		=
			\kernel*_\varepsilon(v_\varepsilon)
			+
			\kernel_\varepsilon(\bar x)
			-
			\innprod{\bar x}{v_\varepsilon}.
		\]
		Noting that \(\kernel_\varepsilon\searrow\kernel\) as \(\varepsilon\searrow0\), it follows from \cite[Prop. 13.16(ii) and 13.28(i)]{bauschke2017convex} that \(\kernel*_\varepsilon\nearrow\kernel*\).
		By virtue of \cite[Prop. 7.4(d)]{rockafellar1998variational}, pointwise monotonic convergence implies that the sequence also converges epigraphically in the sense of \cite[Def. 7.1]{rockafellar1998variational}.
		Combined with the fact that \(v_\varepsilon\to v\coloneqq\nabla\kernel(\bar x)+[\nabla f(x)-\nabla f(\bar x)]\), epigraphical convergence entails \(\liminf_{\varepsilon\searrow0}\kernel*_\varepsilon(v_\varepsilon)\geq\kernel*(v)\) by \cite[Prop. 7.2]{rockafellar1998variational}.
		Therefore, taking the limit inferior as \(\varepsilon\searrow0\) in the above inequality yields \eqref{eq:Bcoco}.

	\item
		``\ref{thm:Bcoco:Bcoco} \(\Rightarrow\) \ref{thm:Bcoco:a*FNE}''
		Denote \(T\coloneqq\nabla f\circ\nabla\kernel*\), \(\eta\coloneqq\nabla\kernel(y)\), and \(\bar\eta\coloneqq\nabla\kernel(\bar y)\) for brevity.
		By summing the two inequalities obtained by interchanging \(\bar x=\bar y\) and \(x=y\) in \eqref{eq:Bcoco}, we obtain that
		\begin{align*}
			\DD_f(\bar y,y)
		\geq{} &
			\D*(\bar\eta+\nabla f(y)-\nabla f(\bar y),\bar\eta)
			+
			\D*(\eta+\nabla f(\bar y)-\nabla f(y),\eta)
		\\
		={} &
			\kernel*(\bar\eta+T\eta-T\bar\eta)
			-
			\kernel*(\bar\eta)
			+
			\kernel*(\eta+T\bar\eta-T\eta)
			-
			\kernel*(\eta)
			+
			\DD_f(\bar y,y).
		\end{align*}
		Canceling out the \(\DD_f(\bar y,y)\) terms and suitably rearranging yields the claimed \a*-FNE inequality \eqref{eq:aFNE}.

	\item
		``\ref{thm:Bcoco:a*FNE} \(\Rightarrow\) \ref{thm:Bcoco:Bsmooth}''
		Follows from \cref{thm:aFNE:Tmonotone}, which implies that \(\nabla f\) and \(\nabla\kernel-\nabla f\) are monotone mappings.
	\qedhere
	\end{itemize}
\end{proof}

Akin to \cref{fact:classic}, we shall see in \cref{thm:DFNE} that monotonicity of \(\psubdiff f\), convexity of \(f\), and \B-firm nonexpansiveness of \(\prox_{\lambda f}\) coincide.
It is certainly easy to see that \(\prox_{\lambda f}\) is \B-FNE provided that \(f\) is convex.
The converse direction, although expected, is more challenging than it seems.
Even in the Euclidean case, to the best of our knowledge, it was not clear whether firm nonexpansiveness of the proximal operator implies convexity, until it was recently shown by Rockafellar in \cite[Thm. 1]{rockafellar2021characterizing}.

\begin{theorem}\label{thm:DFNE}%
	Suppose that \(\kernel\) is Legendre and 1-coercive, and that \cref{ass:range} holds for a proper, lsc, and \(\kernel\)-prox-bounded function \(\func{f}{\X}{\Rinf}\)
	and \(\lambda\in(0,\pb)\).
	Then, the following are equivalent:
	\begin{enumerateq}
	\item \label{thm:DFNE::convex}%
		\(f\) is convex.
	\item \label{thm:DFNE::convex*}%
		\(f\restr_{\interior\X}\) is convex.
	\item \label{thm:DFNE::mono}%
		\(\psubdiff f\) is monotone.
	\item \label{thm:DFNE::maxmono}%
		\(\psubdiff f\) is maximally monotone relative to \(\interior\X\); that is, no monotone operator \(\ffunc{T}{\interior\X}{\R^n}\) exists such that \(\graph\psubdiff f\subsetneq\graph T\).
	\item \label{thm:DFNE::DFNE}%
		\(
			\innprod{x_1-x_2}{\nabla\kernel(y_1)-\nabla\kernel(y_2)}
		\geq
			\DD(x_1,x_2)
		\)
		for every \(y_i\in\Y\) and \(x_i\in\prox_{\lambda f}(y_i)\), \(i=1,2\);
		that is, \(\prox_{\lambda f}\) is (single-valued and) \B-FNE.
	\end{enumerateq}
	Any of the above equivalent conditions implies that \(\prox_{\lambda f}\circ\nabla\kernel*\) is maximally monotone.
\end{theorem}
\begin{proof}
	In the following discussion, observe that because of \cref{ass:range} (together with 1-coercivity and \(\kernel\)-prox-boundedness that ensure \(\range\prox_{\lambda f}\neq\emptyset\)) it necessarily holds that \(\dom f\cap\interior\X\neq\emptyset\); see \cref{thm:mainprop}.
	\begin{itemize}
	\item
		``\ref{thm:DFNE::convex} \(\Rightarrow\) \ref{thm:DFNE::convex*}'' and ``\ref{thm:DFNE::maxmono} \(\Rightarrow\) \ref{thm:DFNE::mono}''
		Trivial.

	\item
		``\ref{thm:DFNE::convex*} \(\Rightarrow\) \ref{thm:DFNE::mono}''
		For any \(x\in\interior\X\) it follows from \cref{thm:FPsubdiff:f} that \(\psubdiff f(x)=\lsubdiff f(x)=\lsubdiff\bigl[f\restr_{\interior\X}\bigr](x)\), where the second equality follows from the fact that \(\interior\X\) is open.
		Monotonicity of \(\psubdiff f\) then follows from that of \(\lsubdiff\bigl[f\restr_{\interior\X}\bigr]\).

	\item
		``\ref{thm:DFNE::mono} \(\Rightarrow\) \ref{thm:DFNE::maxmono}''
		Take \((\bar x,\bar u)\in\interior\X\times\R^n\) which is monotonically related to \(\graph\psubdiff f\), that is, such that for every \((x,u)\in\graph\psubdiff f\)
		\begin{equation}\label{eq:mono related}
			\innprod{x-\bar x}{u-\bar u}\geq0.
		\end{equation}
		We claim that \(\bar u\in\psubdiff f(\bar x)\), which proves the desired maximality.
		To see this, take
		\(
			\hat x
		\in
			\prox_{\lambda f}\circ\nabla\kernel*(\lambda\bar u+\nabla\kernel(\bar x))
		=
			\prox_{\lambda f}\circ\nabla\kernel*(\lambda\bar u+\nabla\kernel(\bar x))\cap\interior\X
		\),
		where the equality holds due to \cref{ass:range}, to which \cref{thm:prox-J} yields
		\(
			\bar u+\lambda^{-1}(\nabla\kernel(\bar x)-\nabla\kernel(\hat x))
		\in
			\psubdiff f(\hat x)
		\) and \eqref{eq:mono related} entails that
		\(
		\innprod{\hat x-\bar x}{\nabla\kernel(\bar x)-\nabla\kernel(\hat x)}\geq0
		\),
		enforcing \(\bar x=\hat x\).
		Invoking \cref{thm:psubdiffchar} to the inclusion
		\(
		\bar x\in\prox_{\lambda f}\circ\nabla\kernel*(\lambda\bar u+\nabla\kernel(\bar x))
		\)
		yields \(\bar u\in\psubdiff f(\bar x)\), justifying the claim.

	\item
		``\ref{thm:DFNE::mono} \(\Rightarrow\) \ref{thm:DFNE::DFNE}''
		Fix \(y_i\in\Y\) and \(x_i\in\prox_{\lambda f}(y_i)\) for \(i=1,2\).
		Since \(x_i\in\interior\X\) by \cref{ass:range}, it follows from \cref{thm:prox-J} that
		\(
			\lambda^{-1}\bigl(\nabla\kernel(y_i)-\nabla\kernel(x_i)\bigr)
		\in
			\psubdiff f(x_i)
		\),
		to which the monotonicity of \(\psubdiff f\) implies that
		\[
			\innprod{x_1-x_2}{
				\bigl(\nabla\kernel(y_1)-\nabla\kernel(x_1)\bigr)
				-
				\bigl(\nabla\kernel(y_2)-\nabla\kernel(x_2)\bigr)
			}
		\geq
			0.
		\]
		After rearranging, the claimed inequality is obtained.
		It then follows from \cite[Prop. 3.5(iii)]{bauschke2003bregman} that \(\prox_{\lambda f}\) is indeed single-valued as required in \cref{def:BFNE}, and is thus \B-FNE as claimed; maximal monotonicity is then ensured by \cref{thm:hypo}.

	\item
		``\ref{thm:DFNE::DFNE} \(\Rightarrow\) \ref{thm:DFNE::convex*}''
		Since \(\prox_{\lambda f}\) is single-valued, \(\lambda f+\kernel\) is convex by virtue of \cref{thm:hypo}.
		Therefore,
		\begin{equation}\label{eq:psubdiffmono}
			\psubdiff f
		=
			\fsubdiff(f+\lambda^{-1}\kernel)\restr_{\interior\X}-\lambda^{-1}\nabla\kernel
		=
			\lsubdiff f\restr_{\interior\X},
		\end{equation}
		where the first identity owes to \cref{thm:FPsubdiff:f}.
		As an intermediate step, we first show that assertion \ref{thm:DFNE::mono} holds true.
		Take \(u_i\in\psubdiff f(x_i)\) for \(i=1,2\).
		Then, \cref{thm:prox-J} entails
		\[
			x_i
		\in
			\prox_{\lambda f}\circ\nabla\kernel*(\lambda u_i+\nabla\kernel(x_i)),
		\]
		from which \B-firm nonexpansiveness yields
		\[
			\innprod{x_1-x_2}{\lambda u_1+\nabla\kernel(x_1)-\lambda u_2-\nabla\kernel(x_2)}
		\geq
			\innprod{x_1-x_2}{\nabla\kernel(x_1)-\nabla\kernel(x_2)}.
		\]
		It follows immediately that \(\psubdiff f\) is monotone, which is assertion \ref{thm:DFNE::mono}.
		Combined with \eqref{eq:psubdiffmono} we deduce that \(\mathring{f}\coloneqq f+\indicator_{\interior\X}\) is proper (because of \cref{ass:range}) and convex (though possibly not lsc).

	\item
		``\ref{thm:DFNE::DFNE} \(\Rightarrow\) \ref{thm:DFNE::convex}''
		As shown in the item above, \(\bar{f}\coloneqq\conv*\mathring{f}\) is proper and convex, and \(\prox_{\lambda f}\) is single-valued.
		For any \(y\in\Y\) we have
		\begin{align*}
			\prox_{\lambda f}(y)
		={} &
			\argmin_{x\in\interior\X}\set{\lambda f(x)+\D(x,y)}
		\\
		={} &
			\argmin_{x\in\X}\set{\lambda\mathring f(x)+\D(x,y)}
		\\
		\subseteq{} &
			\argmin_{x\in\X}\set{\lambda\bar f(x)+\D(x,y)}
			=
			\prox_{\lambda\bar f}(y),
		\end{align*}
		where the first identity holds because of \cref{ass:range}, the second one because \(\mathring{f}\) equals \(f\) on \(\interior\X\) and \(\infty\) elsewhere, and the third one because of \cref{thm:argmincl}.
		On the other hand, both \(\prox_{\lambda f}(y)\) and \(\prox_{\lambda\bar f}(y)\) are singletons (the latter one because \(\bar f\) is convex).
		Therefore, the inclusion holds as equality.
		We then infer that \(\bar f=\hull_\lambda{\bar f}=\hull_\lambda{f}+c=f+c\) for some \(c\in\R\), where the first and last identity follow from the fact that \(\bar f+\lambda^{-1}\kernel\) and \(f+\lambda^{-1}\kernel\) are (proper, lsc and) convex, cf. \cref{thm:Euclidean}, and the second one from \cref{thm:coincidence}.
		Therefore, \(f\) coincides with the convex function \(\bar f-c\) (in fact, with \(c=0\)).
	\qedhere
	\end{itemize}
\end{proof}

\cref{ass:range} is essential for the full equivalence in \cref{thm:DFNE}.
Even more, the necessity of \cref{ass:range} showcases a striking departure from the Euclidean case studied in \cite[Thm. 4.7]{luo2024level}.
When this requirement is not met, it can still be shown by the same arguments that \(\interior\X\cap\prox_{\lambda f}\) being \B-FNE is tantamount to \(\psubdiff f\) being monotone, but not necessarily \emph{maximally} so, as the following example demonstrates.

\begin{example}[nonmaximal monotonicity of \(\psubdiff f\)]\label{ex:psubdiff_notMM}%
	Consider
	\begin{align*}
		\kernel(x)
	={} &
		\begin{cases}
			-\sqrt{1-x^2} & \text{if }|x|\leq1
			\\
			\infty & \text{otherwise,}
		\end{cases}
	\shortintertext{and \(\func{f}{\X=[-1,1]}{\Rinf}\) defined as}
		f(x)
	={} &
		\begin{cases}
			\sqrt{\tfrac{1}{4}-(x-\tfrac{1}{2})^2} & \text{if } |x-\tfrac{1}{2}|\leq\tfrac{1}{2}
			\\
			\infty & \text{otherwise.}
		\end{cases}
	\end{align*}
	\begin{minipage}[b]{0.55\linewidth}%
		For \(\lambda=2\), \cref{thm:FPsubdiff:f} yields that
		\begin{align*}
			\psubdiff f(x)
		={} &
			\mathrlap{
				\bigl[\fsubdiff(f+\lambda^{-1}\kernel)\bigr]\restr_{(-1,1)}(x)
				-
				\lambda^{-1}\nabla\kernel(x)
			}
		\\
		={} &
			\begin{cases}
				\left(-\infty,\tfrac{1}{2}\right] & \text{if } x=0\\
				\emptyset & \text{otherwise.}
			\end{cases}
		\end{align*}
		Apparently, \(\psubdiff f\) is monotone but not maximally so.
		This is because
		\[
			\range\prox_{2f}=\set{0,1}\not\subseteq(-1,1)=\interior\X,
		\]
		and therefore \cref{ass:range} is violated.
	\end{minipage}
	\hfill
	\begin{minipage}[b]{0.44\linewidth}%
		\vspace*{0pt}%
		\includegraphics[width=\linewidth]{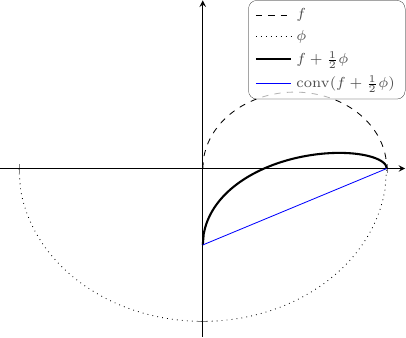}%
		\vspace*{0pt}%
	\end{minipage}
\end{example}

In variational analysis, differentiability of a function is closely tied with a two-sided condition on a subdifferential.
For example, a function \(f\) is differentiable at \(\bar x\) iff \(\rsubdiff(\pm f)(\bar x)\neq\emptyset\).
We end this section investigating \B-smoothness, which plays a central role in numerical algorithms without the classical Lipschitz smoothness; see for instance \cite{bauschke2017descent,lu2018relatively,ahookhosh2021bregman,bolte2018first,maddison2021dual} and the references therein.
It turns out that \B-smoothness can be characterized by a two-sided condition on the Bregman level proximal subdifferential.

\begin{lemma}\label{thm:pwBsmooth}%
	For any \(\func{f}{\X}{\R}\), \(\bar x\in\interior\X\), and \(L>0\) the following are equivalent:
	\begin{enumerateq}
	\item \label{thm:pwBsmooth:nonempty}%
		\(\psubdiff[]^{L\kernel}(\pm f)(\bar x)\neq\emptyset\).
	\item \label{thm:pwBsmooth:singleton}%
		\(f\) is differentiable at \(\bar x\) and \(\psubdiff[]^{L\kernel}[\pm f](\bar x)=\set{\pm\nabla f(\bar x)}\).
	\item \label{thm:pwBsmooth:D}%
		\(f\) is differentiable at \(\bar x\) and \(\D_{L\kernel\pm f}(x,\bar x)\geq0\) \(\forall x\in\X\).
	\end{enumerateq}
\end{lemma}
\begin{proof}~
	\begin{itemize}
	\item ``\ref{thm:pwBsmooth:nonempty} \(\Rightarrow\) \ref{thm:pwBsmooth:singleton}''
		If follows from \cref{thm:FPsubdiff:f} that \(\rsubdiff (\pm f)(\bar x)\neq\emptyset\), implying differentiability of \(f\) at \(\bar x\) with \(\rsubdiff[\pm f](\bar x)=\set{\pm\nabla f(\bar x)}\).
		Then, from the inclusions
		\[
			\emptyset
		\neq
			\psubdiff[]^{L\kernel}[\pm f](\bar x)
		\subseteq
			\rsubdiff[\pm f](\bar x)
		=
			\set{\pm\nabla f(\bar x)}
		\]
		it necessarily follows that
		\(
			\psubdiff[]^{L\kernel}[\pm f](\bar x)
		=
			\set{\pm\nabla f(\bar x)}
		\).

	\item ``\ref{thm:pwBsmooth:singleton} \(\Rightarrow\) \ref{thm:pwBsmooth:D}''
		Follows by the definition of \(\psubdiff[]^{L\kernel}[\pm f](\bar x)\), cf. \eqref{eq:psubdiff}.

	\item ``\ref{thm:pwBsmooth:D} \(\Rightarrow\) \ref{thm:pwBsmooth:nonempty}''
		Again by the definition of \(\psubdiff[]^{L\kernel}[\pm f](\bar x)\), it follows that \(\pm\nabla f(\bar x)\in\psubdiff[]^{L\kernel}[\pm f](\bar x)\), hence that \(\psubdiff[]^{L\kernel}[\pm f](\bar x)\) are nonempty.
	\qedhere
	\end{itemize}
\end{proof}

\begin{theorem}\label{thm:Bsmooth}%
	Let \(\func{\kernel}{\R^n}{\Rinf}\) be 1-coercive and Legendre, and let \(\func{f}{\X}{\R}\) and \(L>0\) be fixed.
	If \(\psubdiff[]^{L\kernel}[\pm f]\neq\emptyset\) on \(\interior\X\), then \(L^{-1}f\) is \B-smooth.

	The converse implication also holds provided that
	\begin{equation}\label{eq:subdiff*:CQfull}
		f(x)
	=
		\lim_{\interior\X\ni x'\to x}f(x')
	\quad
		\forall x\in\X\setminus\interior\X.
	\end{equation}
\end{theorem}
\begin{proof}
	Suppose first that \(\psubdiff[]^{L\kernel}[\pm f]\neq\emptyset\) on \(\interior\X\).
	Then \cref{thm:pwBsmooth} implies that \((\forall x,x'\in\interior\X)\) \(\D_{L\kernel\pm f}(x',x)\geq0\) and equivalently \(L\kernel\pm f\) is convex on \(\interior\X\).
	Conversely, the \B-smoothness conditions enforce that \(L\kernel\pm f+\indicator_{\interior\X}\) are convex, which combined with the additional assumption \eqref{eq:subdiff*:CQfull} implies that \(L\kernel\pm f=\closure(L\kernel\pm f+\indicator_{\interior\X})\) are convex.
	In turn, to each \(x\in \interior\X=\relint\dom(L\kernel\pm f)\) there correspond some \(u^{\pm}\in\fsubdiff(L\kernel\pm f)(x)\) such that
	\[
		(\forall x'\in\X)~
		(L\kernel\pm f)(x')\geq (L\kernel\pm f)(x)+\innprod{u^{\pm}}{x'-x},
	\]
	which amounts to \(u^{\pm}-L\nabla\kernel(x)\in\psubdiff[]^{L\kernel}[\pm f](x)\).
\end{proof}

\begin{remark}
	The condition \eqref{eq:subdiff*:CQfull} is necessary.
	Consider
	\begin{align*}
		\kernel(x)
	=
		\begin{cases}
			-\sqrt{1-x^2}
		&
			\text{if }|x|\leq1
		\\
			\infty
		&
			\text{otherwise},
		\end{cases}
	\end{align*}
	and \(\func{f}{\X\mathrel{{=}}[-1,1]}{\Rinf}\) defined by \(f(x)=\sqrt{1-x^2}\) for \(|x|<1\) and \(f(x)=-1\) for \(|x|=1\).
	Then \(f\) is \B-smooth relative to \(\kernel\) because both \(\kernel\pm f\) are convex on  \(\interior\X=(-1,1)\).
	However, \(\conv(f+\kernel)(x)=-1<0=(f+\kernel)(x)\) for all \(|x|<1\), which means that \(\psubdiff f(x)=\emptyset\) for every \(x\in(-1,1)\) in view of \cref{thm:psubdiffchar}.
	This is because
	\[
	\lim_{x\to \pm 1}f(x)=0\neq f(\pm 1),
	\]
	violating the condition \eqref{eq:subdiff*:CQfull}.
\end{remark}

In the classic setting, a function \(f\) is convex iff so is its Moreau envelope; see \cite[Thm. 3.17]{wang2010chebyshev}.
For Bregman Moreau envelopes, although it has been observed that \(\env{f}\circ\nabla\kernel*\) may not be convex even when \(f\) is so, it remains unclear when and why this should happen.

\Cref{thm:envcvx} below is dedicated to characterizing the convexity of \(\env{f}\circ\nabla\kernel*\), and extends the correspondence between \cref{fact:classic:envcvx,fact:classic:envcoco}.
Unlike the Euclidean setting, in the Bregman framework convexity of \(f\) and of its envelope are distinct properties, each corresponding to a different flavor of firm nonexpansiveness of the proximal operator; cf. \cref{fig:diagram}.

\begin{theorem}\label{thm:envcvx}%
	Suppose that \(\kernel\) is Legendre, and let a proper, lsc, and \(\kernel\)-prox-bounded function \(\func{f}{\X}{\Rinf}\) be fixed.
	Then, denoting \(\func{h_\lambda\coloneqq\env{f}\circ\nabla\kernel*}{\Y*}{\Rinf}\), for any \(\lambda\in(0,\pb)\) the following are equivalent:
	\begin{enumerateq}
	\item \label{thm:envcvx::envcvx}%
		\(h_\lambda\) is proper and convex.
	\item \label{thm:envcvx::envB*smooth}%
		\(\lambda h_\lambda\) is convex and \B*-smooth.\footnote{%
			Strictly speaking, \B*-smoothness presupposes a function defined on \(\X*\), whereas \(\lambda h_\lambda\) is only defined on its interior.
			We shall nonetheless appeal to \cref{rem:Bcvx:lsc} and treat this as a slight notational abuse with no loss of generality.
		}%
	\item \label{thm:envcvx::fB*smooth}%
		\((\lambda\tilde f+\kernel)^*\restr_{\X*}\) is \B*-smooth, where \(\func{\tilde f}{\R^n}{\Rinf}\) is the canonical extension of \(f\) as in \cref{def:fext}.
	\item \label{thm:envcvx::envcvxC1}%
		\(h_\lambda\) is convex and continuously differentiable.
	\end{enumerateq}
	Under any of the above conditions, one has that
	\(
		\nabla h_\lambda
	=
		\lambda^{-1}\bigl(\nabla\kernel*-\prox_{\lambda f}\circ\nabla\kernel*\bigr)
	\),
	and in particular \(\prox_{\lambda f}\) is single-valued.
	Moreover, the following conditions also hold:
	\begin{enumerateq}[resume]
	\item \label{thm:envcvx::proxaFNE}%
		\(\prox_{\lambda f}\) is \a-FNE.
	\item \label{thm:envcvx::Lip}%
		\(\dom\prox_{\lambda f}=\Y\), and
		\(
			\innprod{x_1-x_2}{\nabla\kernel(y_1)-\nabla\kernel(y_2)}
		\leq
			\DD(y_1,y_2)
		\)
		for all \(y_i\in\Y\) and \(x_i\in\prox_{\lambda f}(y_i)\).
	\end{enumerateq}
	When \(\kernel\) is 1-coercive, all the listed statements are equivalent and entail \B-weak convexity of \(f\).
\end{theorem}
\begin{proof}
	We begin by recalling the identity
	\(
		\lambda h_\lambda
	=
		\kernel*\restr_{\Y*}
		-
		(\lambda\tilde f+\kernel)^*\restr_{\Y*}
	\)
	of \cref{thm:envconj}.
	When \(h_\lambda\) is proper, it is finite valued and thus the identity can be rearranged as
	\begin{equation}\label{eq:envconj}
		\kernel*\restr_{\Y*}
	=
		\lambda h_\lambda
		+
		(\lambda\tilde f+\kernel)^*\restr_{\Y*}
	\end{equation}
	as no ``\(\infty-\infty\)'' indeterminate occurrences arise.
	Moreover, when \(h_\lambda\) is differentiable the claimed formula of the gradient directly follows from \cref{thm:lsubdiffenv}.
	\begin{itemize}
	\item
		``\ref{thm:envcvx::envcvx} \(\Rightarrow\) \ref{thm:envcvx::envB*smooth}'' and ``\ref{thm:envcvx::envB*smooth} \(\Leftrightarrow\) \ref{thm:envcvx::fB*smooth}''
		Follow from \eqref{eq:envconj}.

	\item
		``\ref{thm:envcvx::envB*smooth} \(\Rightarrow\) \ref{thm:envcvx::envcvx}''
		and
		``\ref{thm:envcvx::envcvxC1} \(\Rightarrow\) \ref{thm:envcvx::envcvx}''
		Obvious.

	\item
		``\ref{thm:envcvx::envcvx} \(\Rightarrow\) \ref{thm:envcvx::envcvxC1}''
		Since \(\kernel*\) is (finite and) differentiable on the open set \(\Y*\), the calculus rule of \cite[Thm. 23.8]{rockafellar1970convex} implies that both convex functions on the right-hand side too must be differentiable.

	\item ``\ref{thm:envcvx::envB*smooth} \(\Rightarrow\) \ref{thm:envcvx::proxaFNE}''
		It follows from \cref{thm:Bcoco} that \(\nabla(\lambda h_\lambda)\circ\nabla\kernel=\id-\prox_{\lambda f}\) is \a-FNE, hence so is \(\prox_{\lambda f}\) by virtue of \cref{thm:aFNE:I-T}.

	\item ``\ref{thm:envcvx::envcvxC1} \(\Rightarrow\) \ref{thm:envcvx::Lip}''
		Convexity and differentiability of \(h_\lambda\) imply that
		\(
			\nabla h_\lambda
		=
			\lambda^{-1}\bigl(\nabla\kernel*-\prox_{\lambda f}\circ\nabla\kernel*\bigr)
		\)
		is monotone, and that \(\prox_{\lambda f}\) is a single-valued mapping (with full domain).
		Assertion \ref{thm:envcvx::Lip} can equivalently be cast as
		\(
			\innprod{x_1-x_2}{\eta_1-\eta_2}
		\leq
			\DD*(\eta_1,\eta_2)
		\)
		for all \(\eta_i\in\Y*\) and \(x_i\in\prox_{\lambda f}\circ\nabla\kernel*(\eta_i)\), which is precisely monotonicity of \(\nabla\kernel*-\prox_{\lambda f}\circ\nabla\kernel*\), hence of \(\nabla h_\lambda\).

	\end{itemize}
	To conclude, suppose that \(\kernel\) is 1-coercive.
	\begin{itemize}
	\item ``\ref{thm:envcvx::proxaFNE} \(\Rightarrow\) \ref{thm:envcvx::envcvxC1}''
		Single-valuedness of \(\prox_{\lambda f}\) implies via \cref{thm:hypo} that \(\lambda f+\kernel\) is convex.
		By \cref{thm:extlsc:lsc}, \(\lambda\tilde f+\kernel\) is thus convex and lsc, so that the inclusion in \cref{thm:lsubdiffenv} holds as equality.
		In particular, \(\lsubdiff(-h_\lambda)\) is a continuous single-valued mapping, because so is \(\prox_{\lambda f}\) by \cref{thm:aFNE:C0}, hence \(h_\lambda\) is continuously differentiable and \(\nabla(\lambda h_\lambda)=\nabla\kernel*-\prox_{\lambda f}\circ\nabla\kernel*\).
		By appealing to \cref{thm:aFNE:I-T}, the formula of the gradient and \a-firm nonexpansiveness of \(\prox_{\lambda f}\) imply that \(\nabla(\lambda h_\lambda)\circ\nabla\kernel\) is \a-FNE.
		In particular, \((\nabla h_\lambda\circ\nabla\kernel)\circ\nabla\kernel*=\nabla h_\lambda\) is monotone by virtue of \cref{thm:aFNE:Tmonotone}.

	\item ``\ref{thm:envcvx::Lip} \(\Rightarrow\) \ref{thm:envcvx::envcvx}''
		As discussed earlier, the inequality in assertion \ref{thm:envcvx::Lip} is precisely monotonicity of \(\nabla\kernel*-\prox_{\lambda f}\circ\nabla\kernel*\), and in particular implies monotonicity of \(\lsubdiff h_\lambda\) by virtue of \cref{thm:upperC1}.
		Thus, \(h_\lambda\) is convex by virtue of \cite[Thm. 12.17]{rockafellar1998variational}.
		\qedhere
	\end{itemize}
\end{proof}

The relation between the convexity of \(\env{f}\circ\nabla\kernel*\) and that of \(f\) will be investigated in \cref{sec:dualitygap}.
Here we first derive sufficient conditions under which the obscure convexity of \(\env{f}\circ\nabla\kernel*\) holds.
The following known result will be useful soon; a proof can be found in, e.g., \cite[Thm. 1.14]{themelis2018proximal}.

\begin{fact}\label{thm:hypoprox}%
	Let \(\func{h}{\R^n}{\Rinf}\) be proper, lsc, and suppose that \(h+\sigma_h \j\) is convex for some \(\sigma_h\in\R\).
	Then, \(h\) is \(\j\)-prox-bounded with threshold \(\lambda_h\geq\max\set{0,\sigma_h}^{-1}\) (with the convention \(0^{-1}=\infty\)).
	Moreover, for any \(\lambda>0\) and \(x_i\in\Eprox_{\lambda f}(y_i)\), \(i=1,2\), it holds that
	\[
		\innprod{x_1-x_2}{y_1-y_2}\geq(1-\lambda\sigma_f)\norm{x_1-x_2}^2.
	\]
	In particular, if \(1-\lambda\sigma_f>0\) then the proximal mapping \(\Eprox_{\lambda f}\) is single-valued and \((1-\lambda\sigma_f)^{-1}\)-Lipschitz continuous.
\end{fact}

\begin{theorem}\label{thm:envcvx:muL}%
	Suppose that \(\kernel\) is 1-coercive and Legendre, and that \(\nabla\kernel\) is \(L\)-Lipschitz continuous for some \(L>0\) (in particular, \(\dom\kernel=\R^n\)).
	Let \(\lambda>0\) and \(\func{f}{\R^n}{\Rinf}\) be proper, lsc, and such that \(\lambda f+\kernel\) is \(L\)-strongly convex
	(as, for instance, when \(f\) is \(\mu\)-strongly convex and \(\lambda\geq\mu^{-1}L\)).%
	Then, \(\lambda<\pb\) and \(\env{f}\circ\nabla\kernel*\) is convex.
\end{theorem}
\begin{proof}
	For simplicity, let \(q\coloneqq\lambda^{-1}(\kernel-\j)\) and \(h\coloneqq f+q\).
	For \(\sigma_h\coloneqq\lambda^{-1}(1-L)\),
	\(
		h+\sigma_h\j
	=
		f+q+\lambda^{-1}(1-L)\j
	=
		\lambda^{-1}
		(\lambda f+\kernel-L\j)
	\)
	is convex owing to the assumption.
	Moreover, \cref{thm:Euclidean} implies that \(\Eprox_{\lambda h}=\prox_{\lambda f}\circ\nabla\kernel*\) is well defined with \(\dom\Eprox_{\lambda h}=\dom\prox_{\lambda f}\circ\nabla\kernel*=\dom\nabla\kernel*.\)
	Hence all conditions of \cref{thm:hypoprox} are satisfied, which yields that \(\Eprox_{\lambda h}\) is Lipschitz with modulus \(L^{-1}\).

	Write \(g_{\lambda}\coloneqq\env{f}\circ\nabla\kernel*\).
	Then \cref{thm:Euclidean} implies that
	\(
	\lambda g_{\lambda}=\lambda\Eenv_\lambda h+\kernel*-\j
	\).
	We claim that \(\Eenv_\lambda h\) is differentiable, which implies
	\(
		\lambda\nabla g_{\lambda}
	=
		\lambda\nabla\Eenv_\lambda h+\nabla\kernel*-\id
	=
		\id-\Eprox_{\lambda h}+\nabla\kernel*-\id
	=
		\nabla\kernel*-\Eprox_{\lambda h}
	\),
	where \(\nabla\kernel*\) is \(L^{-1}\)-strongly monotone.
	Indeed, \(h+\lambda^{-1}\j=f+\lambda^{-1}\kernel\) is \(\lambda^{-1}L\)-strongly convex and in particular strictly convex.
	We claim that \(\lambda_h\), the \(\j\)-prox-boundedness threshold of \(h\), satisfies \(\lambda_h>\lambda\), making \cite[Thm. 3.5]{wang2010chebyshev} applicable and thus furnishing differentiability of \(\Eenv_\lambda h\).
	When \(L\geq1\), appealing to \cref{thm:hypoprox} and the convexity of \(h+\sigma_h\j\) yields \(\lambda_h=\infty\).
	When \(L<1\), the function \(h+(\lambda')^{-1}\j\) is strongly convex provided that \(\lambda<\lambda'<\lambda/(1-L)\) and therefore bounded below, justifying that \(\lambda_h\geq\lambda/(1-L)>\lambda\) in view of \cite[Ex. 1.24]{rockafellar1998variational}.

	Altogether, for every \(\eta_i\in\R^n\) and \(x_i=\Eprox_{\lambda h}(\eta_i)\),
	\begin{align*}
	&
		\lambda\innprod{\nabla g_{\lambda} (\eta_1)-\nabla g_{\lambda}(\eta_2)}{\eta_1-\eta_2}
	\\
	={} &
		\innprod{\nabla\kernel*(\eta_1)-\nabla\kernel*(\eta_2)}{\eta_1-\eta_2}
	-
		\innprod{x_1-x_2}{\eta_1-\eta_2}
	\\
	\geq{} &
		L^{-1}\norm{\eta_1-\eta_2}^2-L^{-1}\norm{\eta_1-\eta_2}^2
	=
		0,
	\end{align*}
	justifying the desired convexity of \(g_{\lambda}\).
\end{proof}

A particularly illuminating special case of the above result arises when \(\kernel\) amounts to a quadratic perturbation of a Lipschitz-smooth (i.e., \B_\j-smooth) function.
In that setting, the envelope reduces to the so-called \emph{forward-backward envelope}, and one recovers the convexity claims drawn in \cite{themelis2020new}.

\begin{example}[Forward-backward and DC envelopes]\label{ex:FBE}%
	The problem of minimizing a composite function \(\varphi=f+g\), where \(f\) has an \(L_f\)-Lipschitz continuous gradient and \(g\) is a (possibly) nonsmooth ``prox-friendly'' term, is a classical and extensively studied setting in optimization that continues to attract significant attention.
	In this context, for any stepsize \(\lambda\in(0, L_f^{-1})\), the \emph{forward–backward envelope}
	\[
		\textstyle
		\FB(y)
	\coloneqq
		\inf_x\set{
			f(y)+\innprod{\nabla f(y)}{x-y}+g(x)+\tfrac{1}{2\lambda}\norm{x-y}^2
		}
	\]
	constitutes a regularized objective that shares global and local minimizers with \(\kernel\), thus providing a convenient analytical device both for convergence analysis and for algorithmic design \cite{liu2017further,themelis2018forward,giselsson2018envelope,themelis2019acceleration,latafat2021block}.
	Even when both \(f\) and \(g\) are convex, the function \(\FB\) need not be convex unless \(f\) is quadratic \cite[Prop. 4.4]{giselsson2018envelope}.
	Convexity is nevertheless recovered upon composition with the bi-Lipschitz mapping \(\Eprox_{-\lambda f}\), leading to the so-called \emph{DC envelope} \(\DC\coloneqq\FB\circ\Eprox_{-\lambda f}\) \cite[Lem. 6]{themelis2020new}.

	These observations admit a transparent explanation via \cref{thm:envcvx:muL}.
	Indeed, one may identify \(\FB=\env{\varphi}\) and \(\DC = \env{\varphi}\circ\nabla\kernel^*\) by considering the \(\mu\)-strongly convex and \(L\)-smooth dgf \(\kernel=\j-\lambda f\), with \(\mu=1-\lambda L_f\) and \(L=1\), for which \(\Eprox_{-\lambda f}\) coincides with \(\nabla \kernel^*\).
	In turn, the differentiability of \(\DC\) \cite[Lem. 4(i)]{themelis2020new}, independent of convexity of \(f\) and despite the potential lack thereof for \(\FB\), is consistent with the conclusions of \cref{thm:envcvx}.
\end{example}

\subsection{A duality gap revisited}\label{sec:dualitygap}
In the Euclidean case, the Moreau envelope is convex if and only if so is the underlying function; see \cref{fact:classic}.
Such equivalence hinges on the well-known duality between Lipschitz smoothness and strong convexity.
Indeed, recall that \(\lambda\Eenv_\lambda{f}=\j-(\lambda f+\j)^*\) and therefore
\begin{align*}
	\Eenv_\lambda f\text{ is convex }
& \Leftrightarrow
	\j-(\lambda f+\j)^*\text{ is convex }\\
& \Leftrightarrow
	(\lambda f+\j)^*\text{ is \(1\)-Lipschitz smooth}\\
& \Leftrightarrow
	\lambda f+\j\text{ is \(1\)-strongly convex}\\
& \Leftrightarrow
	f \text{ is convex}.
\end{align*}
However, Laude et al. developed a conjugacy theory in \cite{laude2023dualities} revealing that the classic duality between strong convexity and Lipschitz smoothness is not inherited by their Bregman extensions, respectively being \emph{\B-strong convexity} and \emph{\B*-smoothness} as in \cref{def:Bcvx}.
More precisely, for a Legendre dgf \(\func{\kernel}{\R^n}{\Rinf}\) and a function \(\func{g}{\R^n}{\Rinf}\),
\[
	g\text{ is \B-strongly convex}
\not\Leftrightarrow
	g^*\text{ is \B*-smooth}.
\]
Adopting these jargons, one has\footnote{%
	Strictly speaking, the validity of the first ``\(\Leftarrow\)'' implication presupposes some consistency between the values of the function at boundary points and on the interior of \(\X\); see \cref{rem:Bcvx:lsc}.
}%
\begin{align*}
	f \text{ is convex}
&\Leftrightarrow
	\lambda f+\kernel\text{ is \B-strongly convex}\\
&\not\Leftrightarrow
	(\lambda f+\kernel)^*\text{ is \B*-smooth}\\
&\Leftrightarrow
	\env{f}\circ\nabla\kernel*\text{ is convex}.
\end{align*}
This shows that the gap between the convexity of \(\env{f}\circ\nabla\kernel*\) and that of \(f\) owes exactly to the duality gap between \B-strong convexity and \B*-smoothness.
In \cref{sec:FNE}, we further examined this phenomenon from an operator perspective.
Specifically, we introduced a new notion of \emph{anisotropic} firm nonexpansiveness and traced the discrepancy back to whether the proximal mapping \(\prox_{\lambda f}\) is \emph{Bregman} or \emph{anisotropically} FNE: in summary, when \(\kernel\) is Legendre and 1-coercive,
\begin{align*}
	f \text{ is convex}
&\Leftrightarrow
	\prox_{\lambda f}\text{ is \B-FNE}\\
&\not\Leftrightarrow
	\prox_{\lambda f}\text{ is \a-FNE}\\
&\Leftrightarrow
	\env{f}\circ\nabla\kernel*\text{ is convex}.
\end{align*}
Hence, the gap can also be interpreted as the distinction between non-Euclidean extensions of firm nonexpansiveness: \emph{Bregman} (\B) and \emph{anisotropic} (\a).
We now record examples to show that convexity of \(\env{f}\circ\nabla\kernel*\) and that of the underlying function \(f\) are indeed distinct notions, and thus so are \a- and \B-firm nonexpansiveness; see also \cite[Ex. 1.4]{wang2022bregman} for the failure of the implication ``\(f\) convex \(\Rightarrow\) \(\env{f}\) convex''.

\begin{example}[\(\env{f}\circ\nabla\kernel*\) convex \(\not\Rightarrow\) \(f\) convex]\label{ex:env_not_f}%
	Let $f(x)=\tfrac{1}{4}(x-1)^4-\tfrac{1}{4}x^4$ and $\kernel(x)=\tfrac{1}{4}x^4$.
	Then $\env_1{f}\circ\nabla\kernel*(y)=-y$.
	Clearly $\env_1{f}\circ\nabla\kernel*$ is convex, while $f$ is not.
\end{example}

\begin{example}[\(f\) convex \(\not\Rightarrow\) \(\env{f}\circ\nabla\kernel*\) convex]\label{ex:f_not_env}%
	Let $f(x)=x$ and $\kernel(x)=\tfrac{1}{3}|x|^3$.
	Then $\env_1{f}\circ\nabla\kernel*(y)=\tfrac{2}{3}|y|^{\tfrac{3}{2}}-\tfrac{2}{3}|y-1|^{\tfrac{3}{2}}$.
	Clearly $f$ is convex but $\env_1{f}\circ\nabla\kernel*$ is not.
\end{example}

The conjugate of \B*-smoothness is not captured by strong convexity in the ``Bregman'' sense, but rather by a distinct \emph{``anisotropic''} counterpart \cite[Thm. 4.3]{laude2023dualities}; see also \cite[Prop. 5.11]{laude2025anisotropic} and \cite[Thm. 4.10]{themelis2026natural} for developments beyond full domain assumptions.
We remind that, when \(\kernel\) is Legendre and with full domain, a proper and lsc function \(\func{f}{\R^n}{\Rinf}\) is said to be \emph{anisotropically strongly convex relative to \(\kernel\)}, or \emph{\a-strongly convex}, if for every \((\bar x,\bar v)\in\graph\lsubdiff f\) with \(\bar v\in\interior\dom\kernel*\) it holds that
\[
		f(x)
	\geq
		f(\bar x)
		+
		\kernel(x-\bar x+\nabla\kernel*(\bar v))
		-
		\kernel(\nabla\kernel*(\bar v))
	\quad
		\forall x\in\R^n.
\]
We may thus synopsize the main conditions in both \cref{thm:envcvx,thm:DFNE} as follows:%

\begin{corollary}\label{thm:aBFNE}%
	Suppose that \(\kernel\) is Legendre and 1-coercive, and let \(\func{f}{\R^n}{\Rinf}\) be proper, lsc, and \(\kernel\)-prox-bounded.
	Then, for any \(\lambda\in(0,\pb)\) the following statements are equivalent (and entail validity of \cref{ass:range}):
	\begin{enumerateq}
	\item
		\(\dom f\cap\interior\X\neq\emptyset\), and both \(f\) and \(\env{f}\circ\nabla\kernel*\) are convex.
	\item
		\(
			\DD(x_1,x_2)
		\leq
			\innprod{x_1-x_2}{\nabla\kernel(y_1)-\nabla\kernel(y_2)}
		\leq
			\DD(y_1,y_2)
		\)
		holds for all \(y_i\in\Y\) and \(x_i=\prox_{\lambda f}(y_i)\), \(i=1,2\).
	\item
		\(\prox_{\lambda f}\) is both \B- and \a-FNE.
	\end{enumerateq}
	If in addition \(\dom\kernel=\R^n\), then the following can be added to the equivalence:
	\begin{enumerateq}[resume]
	\item
		\(\lambda f+\kernel\) is both \B- and \a-strongly convex.
	\end{enumerateq}
\end{corollary}

	\section{Conclusion and future work}\label{sec:conclusion}
		The contribution of this paper is twofold.
On one hand, we proposed and investigated systematically the left and right Bregman level proximal subdifferentials, which turn out to be useful tools in variational analysis and optimization.
The (warped) resolvents of these new subdifferentials are always the Bregman proximal operators under a standard range assumption, whereas a similar characterization in terms of classical subdifferentials may fail to hold in the absence of convexity.
Among other results, we characterized the existence of the left Bregman level proximal subdifferential, revealing an interesting connection with a ``pointwise'' Bregman weak convexity.
On the other hand, we established equivalent characterizations of key properties of Bregman proximal operators, such as ``Bregman'' firm nonexpansiveness and convex-valuedness, in terms of the underlying function and its associated Bregman level proximal subdifferential.
We further introduced a complementary ``anisotropic'' notion of firm nonexpansiveness, which allowed us to reinterpret the conjugate duality gap documented in \cite{laude2023dualities} from an operator perspective, as well as in terms of convexity of a function and of its envelope, and to establish a novel characterization of relative smoothness for convex functions under minimal assumptions.

There are many open avenues worth devoting future investigations.
For instance, in the Euclidean setting the level proximal subdifferential plays a central role in characterizing the variational convexity of functions and local nonexpansiveness of proximal operators; see \cite{luo2024level}.
It is thus tempting to see whether similar connections exist in the more general Bregman setting.
Another interesting question is whether a complete picture regarding the relation between the convexity of the Bregman Moreau envelope and that of the underlying function can be derived.

	\paragraph*{Acknowledgments}
		We are grateful to Prof. Xianfu Wang for his insightful comments and suggestions, which improved the quality of this work significantly.
		We also sincerely thank the anonymous referees for their constructive criticism and for bringing to our attention several issues in our original draft.


	\phantomsection
	\addcontentsline{toc}{section}{References}%
	\bibliographystyle{plain}
	\bibliography{TeX/references}

\end{document}